\documentclass[12pt, twoside]{article}
\usepackage{amsmath,amsthm,amssymb}
\usepackage{times}
\usepackage{enumerate}

\def\titlerunning#1{\gdef\titrun{#1}}
\makeatletter
\def\author#1{\gdef\autrun{\def\and{\unskip, }#1}\gdef\@author{#1}}
\def\address#1{{\def\and{\\\hspace*{18pt}}\renewcommand{\thefootnote}{}%
\footnote {#1}}%
\markboth{\autrun}{\titrun}}
\makeatother
\def\email#1{e-mail: #1}
\def\subjclass#1{{\renewcommand{\thefootnote}{}%
\footnote{\emph{Mathematics Subject Classification (2010):} #1}}}
\def\keywords#1{\par\medskip
\noindent\textbf{Keywords.} #1}

\newtheorem{theorem}{Theorem}[section]
\newtheorem{corollary}[theorem]{Corollary}
\newtheorem{definition}[theorem]{Definition}

\newtheorem{lemma}[theorem]{Lemma}
\newtheorem{notation}[theorem]{Notation}
\newtheorem{proposition}[theorem]{Proposition}
\newtheorem{remark}[theorem]{Remark}

\begin{document}

%%%%% To ease editing, add:

\baselineskip=17pt

%%%%%%%%%%%%%%%%

%% In the running head, give an abbreviation of the title.
\titlerunning{Calculus of Variations with Differential Forms}

\title{Calculus of Variations with Differential Forms}

\author{Saugata Bandyopadhyay \and Bernard Dacorogna
\and
Swarnendu Sil}

\date{}

\maketitle

\address{S. Bandyopadhyay: IISER Kolkata, Mohanpur Campus, Mohanpur-741252, India; \email{saugata.bandyopadhyay@iiserkol.ac.in}
\and
B. Dacorogna: Section de Math\'ematiques, EPFL, 1015 Lausanne, Switzerland; \email{bernard.dacorogna@epfl.ch}
\and
S. Sil: Section de Math\'ematiques, EPFL, 1015 Lausanne, Switzerland; \email{swarnendu.sil@epfl.ch}}


%%%%%%%%

\begin{abstract}
We study integrals of the form $\int_{\Omega}f\left(  d\omega\right)$,
where $1\leq k\leq n$, $f:\Lambda^{k}\rightarrow\mathbb{R}$ is continuous and $\omega$ is a
$\left(k-1\right)$-form. We introduce the appropriate notions of
convexity, namely ext. one convexity, ext. quasiconvexity and ext.
polyconvexity. We study their relations, give several examples and
counterexamples. We finally conclude with an application to a minimization problem.

%% Keywords are optional
\end{abstract}

\section{Introduction}

In this article, we study integrals of the form $$\int_{\Omega}f\left(  d\omega\right),$$
where $1\leq k\leq n$ are integers, $f:\Lambda^{k}\rightarrow\mathbb{R}$ is a continuous function and $\omega$ is a
$\left(  k-1\right)$-form. When $k=1,$ by abuse of notations identifying
$\Lambda^{1}$ with $\mathbb{R}^{n}$ and the
operator $d$ with the gradient, this is the classical problem of the calculus
of variations where one studies integrals of the form%
\[
\int_{\Omega}f\left(  \nabla\omega\right)  .
\]
This is a \emph{scalar} problem in the sense that there is only one function
$\omega.$ It is well known that in this last case the \emph{convexity} of $f$
plays a crucial role. As soon as $k\geq2,$ the problem is more of a
\emph{vectorial} nature, since then $\omega$ has now several components.
However, it has some special features that a general vectorial
problem does not have. Before going further one should have two examples in mind.\smallskip

1) If $k=2$, with our usual abuse of notations,%
\[
\omega:\mathbb{R}^{n}\rightarrow\mathbb{R}^{n}\text{\quad and\quad}%
d\omega=\operatorname{curl}\omega.
\]

2) If $k=n$, by abuse of notations and up to some changes of signs,%
\[
\omega:\mathbb{R}^{n}\rightarrow\mathbb{R}^{n}\text{\quad and\quad}%
d\omega=\operatorname{div}\omega.
\]
So let us now discuss some specific features of our problem.\smallskip

- The first important point is the lack of \emph{coercivity}. Indeed, even if
the function $f$ grows at infinity as the norm to a certain power, this does
not imply control on the full gradient but only on some combination of it,
namely $d\omega.$ So when looking at minimization problems, this fact requires
special attention (see Theorem \ref{Thm existence de minima}).\smallskip

- From the point of view of convexity, the situation is, in some cases, simpler
than in the general vectorial problem. Indeed consider the above two examples
(with $n=3$ for the first one). Although the problems are vectorial they
behave as if they were scalar (cf. Theorem
\ref{Thm general sur poly quasi ...}).\smallskip

- One peculiarity (cf. Theorem \ref{Thm principal quasiaffine})
that particularly stands out is how the problem changes its behaviour with
a change in the order of the form. When $k$ is odd, or when $2k>n$ (in particular
$k=n$), there is no nonlinear function which is ext. quasiaffine and therefore, the problem behaves as if it were scalar.
However, the situation changes significantly when $k\leq2n$ is even. It turns out that we have an ample supply of nonlinear functions that are ext. quasiaffine in this case. For example, the nonlinear function%
\[
f\left(  \xi\right)  =\left\langle c;\xi\wedge\xi\right\rangle ,
\]
where $c\in\Lambda^{2k}$, is ext. quasiaffine. See Theorem \ref{Thm principal quasiaffine} for the complete characterization of ext. quasiaffine functions which, 
in turn, determines all weakly continuous functions with respect to the $d$-operator, see Bandyopadhyay-Sil \cite{Band-Dac-Sil} for detail.

\smallskip

Because of the special nature of our problem, we are led to introduce the
following terminology:\emph{ ext. one convexity}, \emph{ext. quasiconvexity} and
\emph{ext. polyconvexity}, which are the counterparts of the classical notions of
the vectorial calculus of variations (see, in particular Dacorogna \cite{Dacorogna
2007}), namely rank one convexity, quasiconvexity and polyconvexity. The relations
between these notions (cf. Theorem \ref{Thm general sur poly quasi ...}) as
well as their manifestations on the minimization problem are the subject of the present paper.
Examples and counterexamples are also discussed in details, notably the case
of ext. quasiaffine functions (see Theorem \ref{Thm principal quasiaffine}),
the quadratic case (see Theorem \ref{Thm formes quadratiques}) and a
fundamental counterexample (see Theorem \ref{Thm comme Sverak}) similar to the
famous example of Sverak \cite{Sverak 1992a}.\smallskip

Some of what has been done in this article may also be seen through classical vectorial calculus of variations.
This connection is elaborated and pursued in detail in a forthcoming article, see Bandyopadhyay-Sil \cite{Band-Sil}. However, the case of differential forms in the context of calculus of variations deserves a separate and independent treatment because of its
special algebraic structure which renders much of the calculation intrinsic, natural and coordinate free.\smallskip

We conclude this introduction by pointing out that the results discussed
in this introduction may be interpreted very broadly in terms of the theory of
compensated compactness introduced by Murat and Tartar, see \cite{Murat(1978)},
\cite{Tartar 1978} and see also Dacorogna \cite{Dacorogna 1982}, Robbin-Rogers-Temple
\cite{Robbin-Rogers-Temple}. In particular, our notion of ext. one convexity is
related to the so called condition of convexity in the directions of
the wave cone $\Lambda.$ Our definition of ext. quasiconvexity is related to those of $A$
and $A-B$ quasiconvexity introduced by Dacorogna (cf. \cite{Dacorogna 1982a}
and \cite{Dacorogna 1982}), see also Fonseca-M\"{u}ller \cite{Fonseca-Muller}.

\section{Definitions and main properties}

\subsection{Definitions}

We start with the different notions of convexity and affinity.

\begin{definition}
Let $1\leq k\leq n$ and $f:\Lambda^{k}\rightarrow\mathbb{R}.$\smallskip

\textbf{(i)} We say that $f$ is \emph{ext. one convex}, if the function%
\[
g:t\rightarrow g\left(  t\right)  =f\left(  \xi+t\,\alpha\wedge\beta\right)
\]
is convex for every $\xi\in\Lambda^{k},$ $\alpha\in\Lambda^{k-1}$ and
$\beta\in\Lambda^{1}.$ If the function $g$ is affine we say that $f$ is
\emph{ext. one affine.}\smallskip

\textbf{(ii)} $f$ is said to
be \emph{ext. quasiconvex}, if $f$ is Borel measurable, locally bounded and
\[
\int_{\Omega}f\left(  \xi+d\omega\right)  \geq f\left(  \xi\right)
\operatorname*{meas}\Omega,
\]
for every bounded open set $\Omega\subset\mathbb{R}^n$, $\xi\in\Lambda^{k}$ and $\omega\in W_{0}^{1,\infty}\left(  \Omega;\Lambda^{k-1}\right)$. If
equality holds, we say that $f$ is \emph{ext. quasiaffine.}\smallskip

\textbf{(iii)} We say that $f$ is \emph{ext. polyconvex}, if there exists a
convex function%
\[
F:\Lambda^{k}\times\Lambda^{2k}\times\cdots\times\Lambda^{\left[  n/k\right]
k}\rightarrow\mathbb{R}
\]
such that%
\[
f\left(  \xi\right)  =F\left(  \xi,\xi^{2},\cdots,\xi^{\left[  n/k\right]
}\right),\text{ for all }\xi \in\Lambda^{k}.
 \]
If $F$ is affine, we say that $f$ is \emph{ext. polyaffine.}
\end{definition}
\begin{remark}
\label{Remarque sur les def de convexite}
\textbf{(i)} The ext. stands for
exterior product in the first and third ones and for the exterior derivative
for the second one.\smallskip

\textbf{(ii)} When $k$ is odd (since then $\xi^{s}=0$ for every $s\geq2$) or
when $2k>n$ (in particular, when $k=n$ or $k=n-1$), then ext. polyconvexity is
equivalent to ordinary convexity (see Proposition
\ref{Proposition equiv polyconvexite}).\smallskip

\textbf{(iii)} When $k=1,$ all the above notions are equivalent to the
classical notion of convexity (cf. Theorem
\ref{Thm general sur poly quasi ...}).\smallskip

\textbf{(iv)} As in Proposition 5.11 of \cite{Dacorogna 2007}, it can easily be
shown that if the inequality of ext. quasiconvexity holds for a given bounded
open set $\Omega,$ it holds for any bounded open set.\smallskip

\textbf{(v)} The definition of ext. quasiconvexity is equivalent (as in
Proposition 5.13 of \cite{Dacorogna 2007}) to the following. Let $D=\left(
0,1\right)  ^{n},$ the inequality%
\[
\int_{D}f\left(  \xi+d\omega\right)  \geq f\left(  \xi\right),
\]
holds for every $\xi\in\Lambda^{k}$ and for every $\omega\in W_{per}^{1,\infty}\left(  D;\Lambda^{k-1}\right)$, where
\[
W_{per}^{1,\infty}\left(  D;\Lambda^{k-1}\right)  =\left\{
\omega\in W^{1,\infty}\left(  D;\Lambda^{k-1}\right)  :\omega\text{ is
}1\text{-periodic in each variable}\right\}  .
\]
\end{remark}

\begin{definition}
 Let $0\leq k\leq n$ and $f:\Lambda^{k}
\rightarrow\mathbb{R}.$ The \emph{Hodge transform} of $f$ is the function $f_{\ast}:\Lambda^{n-k}\rightarrow\mathbb{R}$
defined as,\smallskip
\[
f_{\ast}(\xi) = f \left(\ast\xi\right), \text{ for all }\xi \in\Lambda^{n-k}.
 \]
\end{definition}
The notion of Hodge transform allows us to extend the notions of convexity with respect to the interior product and the $\delta$-operator as follows.
\begin{definition}
Let $0\leq k\leq n-1$ and $f:\Lambda^{k}
\rightarrow\mathbb{R}$.We say that \smallskip

(i) $f$ is \emph{int. one convex}, if $f_{\ast}$ is \emph{ext. one convex}.\smallskip

(ii) $f$ is \emph{int. quasiconvex}, if $f_{\ast}$ is \emph{ext. quasiconvex}.\smallskip

(iii) $f$ is \emph{int. polyconvex}, if $f_{\ast}$ is \emph{ext. polyconvex}.
\end{definition}

\begin{remark}
\textbf{(i)} Statements similar to those in Remark \ref{Remarque sur les def de convexite} hold in the case of int. convexity as well.\smallskip

\textbf{(ii)} It is easy to check that $f$ is \emph{int. one convex}, if and only if the function%
\[
g:t\rightarrow g\left(  t\right)  =f\left(  \xi+t\,\beta\,\lrcorner
\,\alpha\right)
\]
is convex for every $\xi\in\Lambda^{k},$ $\beta\in\Lambda^{1}$ and $\alpha
\in\Lambda^{k+1}.$ Furthermore, $f$ is \emph{int. quasiconvex} if and only if $f$ is Borel measurable, locally bounded and
\[
\int_{\Omega}f\left(  \xi+\delta\omega\right)  \geq f\left(  \xi\right)
\operatorname*{meas}\Omega,
\]
for every bounded open set $\Omega\subset\mathbb{R}^n$, $\xi\in\Lambda^{k}$ and $\omega\in W_{0}^{1,\infty}\left(  \Omega;\Lambda^{k+1}\right)$. The
third one is however a little more involved and we leave out the
details.\smallskip

\end{remark}

In what follows, we will discuss the case of ext. convexity only. The case of int. convexity can be handled analogously.

\subsection{Preliminary lemmas}
In this subsection, we state two lemmas which will be used in sequel. See \cite{Sil} for the proofs. We start with the following problem of prescribed differentials. Let us recall that $\alpha\in\Lambda^k$ is said to be \emph{1-divisible} if $\alpha=a\wedge b$, for some $a\in\Lambda^{k-1}$, $b\in \Lambda^1$.
\begin{lemma}\label{dik-form:1}
Let $1\leq k\leq n$ and let $\omega_1,\omega_2\in\Lambda^{k}$. Then, there exists $\omega\in W^{1,\infty}(\Omega;\Lambda^{k-1})$ satisfying
$$
d\omega\in\{\omega_1,\omega_2\},\text{ a.e. in }\Omega,
$$
(and taking both values), if and only if $\omega_1-\omega_2$ is 1-divisible.
\end{lemma}
Using Lemma \ref{dik-form:1}, one can deduce the following approximation lemma for $k$-forms. See Lemma 3.11 of \cite{Dacorogna 2007} for the case of the gradient.
\begin{lemma}\label{approximation theorem}
Let $1\leq k\leq n$,  $t\in [0,1]$ and let $\alpha,\beta\in\Lambda^{k}$ be such that $\alpha\neq\beta$ and $\alpha-\beta$ is 1-divisible. Let $\Omega\subset\mathbb{R}^n$ be open, bounded and let $\omega:\overline{\Omega}\rightarrow\Lambda^{k-1}$ satisfy
$$
d\,\omega=t\alpha+(1-t)\beta,\text{ in }\overline{\Omega}.
$$
Then, for every $\epsilon>0$, there exist $\omega_\epsilon\in\operatorname*{Aff}_{\text piece}\left(\overline{\Omega};\Lambda^{k-1}\right)$ and disjoint open sets $\Omega_\alpha,\Omega_\beta\subset\Omega$ such that
\begin{enumerate}
\item $|\operatorname*{meas}(\Omega_\alpha)-t\operatorname*{meas}(\Omega)|\leqslant\epsilon$ and $|\operatorname*{meas}(\Omega_\beta)-(1-t)\operatorname*{meas}(\Omega)|\leqslant\epsilon$,
\item $\omega_\epsilon=\omega,$ in a neighbourhood of $\partial\Omega$,
\item $\Vert \omega_\epsilon-\omega\Vert_{L^\infty(\overline{\Omega})}\leqslant\epsilon$,
\item $d\omega_\epsilon(x)=\left\{\begin{array}{rl}
\alpha,&\text{if }x\in \Omega_\alpha,\\
\beta,&\text{if }x\in \Omega_\beta,
\end{array}\right.
$
\item $\operatorname*{dist}\left(d\omega_\epsilon(x);\{t\alpha+(1-t)\beta:t\in [0,1]\}\right)\leqslant\epsilon$, for all  $x\in\Omega$ a.e.
\end{enumerate}
\end{lemma}

\subsection{Main properties}

The different notions of convexity are related as follows.

\begin{theorem}
\label{Thm general sur poly quasi ...}Let $1\leq k\leq n$ and $f:\Lambda
^{k}  \rightarrow\mathbb{R}.$\smallskip

\textbf{(i)} The following implications hold%
\[
f\text{ convex }\Rightarrow\text{ }f\text{ ext. polyconvex }\Rightarrow\text{
}f\text{ ext. quasiconvex }\Rightarrow\text{ }f\text{ ext. one convex.}%
\]

\textbf{(ii)} If $k=1,n-1,n$ or $k=n-2$ is odd, then%
\[
f\text{ convex }\Leftrightarrow\text{ }f\text{ ext. polyconvex }%
\Leftrightarrow\text{ }f\text{ ext. quasiconvex }\Leftrightarrow\text{
}f\text{ ext. one convex.}%
\]
Moreover, if $k$ is odd or $2k>n,$ then%
\[
f\text{ convex }\Leftrightarrow\text{ }f\text{ ext. polyconvex.}%
\]

\textbf{(iii)} If either $2\leq k\leq n-3$ or
$k=n-2\geq 2$ is even, then%
\[
f\text{ ext. polyconvex }%
%TCIMACRO{\QATOP{\Rightarrow}{\nLeftarrow}}%
%BeginExpansion
\genfrac{}{}{0pt}{}{\Rightarrow}{\nLeftarrow}%
%EndExpansion
\text{ }f\text{ ext. quasiconvex},
\]
while if $2\leq k\leq n-3$ (and thus $n\geq k+3\geq5$), then%
\[
f\text{ ext. quasiconvex }%
%TCIMACRO{\QATOP{\Rightarrow}{\nLeftarrow}}%
%BeginExpansion
\genfrac{}{}{0pt}{}{\Rightarrow}{\nLeftarrow}%
%EndExpansion
\text{ }f\text{ ext. one convex.}%
\]

\end{theorem}

\begin{remark}
\label{Remarque apresThm general sur poly quasi ...}\textbf{(i)} The last statement in (ii) for $k$ even and $n\geq2k$ is false,
as the following simple example shows. Let $f:\Lambda^{2}\left(
\mathbb{R}^{4}\right)  \rightarrow\mathbb{R}$ be defined by%
\[
f\left(  \xi\right)  =\left\langle e^{1}\wedge e^{2}\wedge e^{3}\wedge
e^{4};\xi\wedge\xi\right\rangle .
\]
The function $f$ is clearly ext. polyconvex but not convex.\smallskip

\textbf{(ii)} The study of
the implications and counter implications for ext. one convexity, ext. quasiconvexity and ext. polyconvexity is therefore complete, except for the last implication, namely%
\[
f\text{ ext. quasiconvex }%
%TCIMACRO{\QATOP{\Rightarrow}{\nLeftarrow}}%
%BeginExpansion
\genfrac{}{}{0pt}{}{\Rightarrow}{\nLeftarrow}%
%EndExpansion
\text{ }f\text{ ext. one convex},
\]
only for the case $k=n-2\geq2$ even (including $k=2$ and $n=4$), which remains
open.\smallskip

\textbf{(iii)} It is interesting to read the theorem when $k=2.$\smallskip

- If $n=2$ or $n=3,$ then%
\[
f\text{ convex }\Leftrightarrow\text{ }f\text{ ext. polyconvex }%
\Leftrightarrow\text{ }f\text{ ext. quasiconvex }\Leftrightarrow\text{
}f\text{ ext. one convex.}%
\]

- If $n\geq4,$ then%
\[
f\text{ convex }%
%TCIMACRO{\QATOP{\Rightarrow}{\nLeftarrow}}%
%BeginExpansion
\genfrac{}{}{0pt}{}{\Rightarrow}{\nLeftarrow}%
%EndExpansion
\text{ }f\text{ ext. polyconvex }%
%TCIMACRO{\QATOP{\Rightarrow}{\nLeftarrow}}%
%BeginExpansion
\genfrac{}{}{0pt}{}{\Rightarrow}{\nLeftarrow}%
%EndExpansion
\text{ }f\text{ ext. quasiconvex.}%
\]

- If $n\geq5,$ then%
\[
f\text{ ext. quasiconvex }%
%TCIMACRO{\QATOP{\Rightarrow}{\nLeftarrow}}%
%BeginExpansion
\genfrac{}{}{0pt}{}{\Rightarrow}{\nLeftarrow}%
%EndExpansion
\text{ }f\text{ ext. one convex},
\]
while the case $n=4$ remains open.
\end{remark}

\begin{proof}
\textbf{(i)} \emph{Step 1.} The implication%
\[
f\text{ convex }\Rightarrow\text{ }f\text{ ext. polyconvex}%
\]
is trivial.\smallskip

\emph{Step 2.} The statement%
\[
f\text{ ext. polyconvex }\Rightarrow\text{ }f\text{ ext. quasiconvex}%
\]
is proved as follows. Observe first that if $\xi\in\Lambda^{k}$ and $\omega\in
W_{0}^{1,\infty}\left(  \Omega;\Lambda^{k-1}\right)  ,$ then%
\begin{equation}
\int_{\Omega}\left(  \xi+d\omega\right)  ^{s}=\xi^{s}\operatorname*{meas}%
\Omega,\quad\text{for every integer }s.
\label{(1) dans le thm sur les implications}%
\end{equation}
We proceed by induction on $s.$ The case $s=1$ is trivial, so we assume that
the result has already been established for $s-1$ and we prove it for $s.$
Note that%
\begin{align*}
\left(  \xi+d\omega\right)  ^{s}  &  =\xi\wedge\left(  \xi+d\omega\right)
^{s-1}+d\omega\wedge\left(  \xi+d\omega\right)  ^{s-1}\smallskip\\
&  =\xi\wedge\left(  \xi+d\omega\right)  ^{s-1}+d\left[  \omega\wedge\left(
\xi+d\omega\right)  ^{s-1}\right]  .
\end{align*}
Integrating, using induction for the first integral and the fact that
$\omega=0$ on $\partial\Omega$ for the second one, we have indeed shown
(\ref{(1) dans le thm sur les implications}). We can now conclude. Since $f$
is ext. polyconvex, we can find a convex function%
\[
F:\Lambda^{k}\times\Lambda^{2k}\times\cdots\times\Lambda^{\left[  n/k\right]
k}\rightarrow\mathbb{R}%
\]
such that%
\[
f\left(  \xi\right)  =F\left(  \xi,\xi^{2},\ldots,\xi^{\left[  n/k\right]
}\right)  .
\]
Using Jensen inequality we find,
\[
\frac{1}{\operatorname*{meas}\Omega}\int_{\Omega}f\left(  \xi+d\omega\right)  \geq F\left(  \frac{1}{\operatorname*{meas}\Omega}\int_{\Omega}\left(
\xi+d\omega\right),\ldots
,\frac{1}{\operatorname*{meas}\Omega}\int_{\Omega}\left(  \xi+d\omega\right)  ^{\left[  n/k\right]  }\right)  .
\]
Invoking (\ref{(1) dans le thm sur les implications}), we have indeed obtained
that%
\[
\int_{\Omega}f\left(  \xi+d\omega\right)  \geq f\left(  \xi\right)
\operatorname*{meas}\Omega,
\]
and the proof of Step 2 is complete.\smallskip

\emph{Step 3.} It follows from Lemma \ref{approximation theorem} that
\[
f\text{ ext. quasiconvex }\Rightarrow\text{ }f\text{ ext. one convex.}%
\]
With Lemma \ref{approximation theorem} at our disposal, the proof is very similar to that of the case of the gradient (cf. Theorem 5.3 in \cite{Dacorogna 2007}) and is omitted. See \cite{Sil} for the details. This concludes the proof of (i).\smallskip

\textbf{(ii)} In all the cases under consideration any $\xi\in\Lambda^{k}$ is 1-divisible (cf. Proposition 2.43 in \cite{Csato-Dac-Kneuss 2012}). Hence, the result%
\[
f\text{ convex }\Leftrightarrow\text{ }f\text{ ext. polyconvex }%
\Leftrightarrow\text{ }f\text{ ext. quasiconvex }\Leftrightarrow\text{
}f\text{ ext. one convex}%
\]
then follows at once. The extra statement (i.e. when $k$ is odd or $2k>n$)%
\[
f\text{ convex }\Leftrightarrow\text{ }f\text{ ext. polyconvex.}%
\]
is proved in Remark \ref{Remarque sur les def de convexite} (ii) and Proposition \ref{Proposition equiv polyconvexite}.\smallskip

\textbf{(iii)} The statement that%
\[
f\text{ ext. polyconvex }%
%TCIMACRO{\QATOP{\Rightarrow}{\nLeftarrow}}%
%BeginExpansion
\genfrac{}{}{0pt}{}{\Rightarrow}{\nLeftarrow}%
%EndExpansion
\text{ }f\text{ ext. quasiconvex},
\]
when $3\leq k\leq n-3$ or $k=n-2\geq4$ is even follows from Theorem
\ref{Thm formes quadratiques} (v) and from Proposition
\ref{Prop k=2 quasic n'implique pas polyc} when $k=2$ and $n\geq4$ (for $k=2$
and $n\geq6,$ we can also apply Theorem \ref{Thm formes quadratiques}
(ii)).\smallskip

The statement that if $2\leq k\leq n-3$ (and thus $n\geq k+3\geq5$), then%
\[
f\text{ ext. quasiconvex }%
%TCIMACRO{\QATOP{\Rightarrow}{\nLeftarrow}}%
%BeginExpansion
\genfrac{}{}{0pt}{}{\Rightarrow}{\nLeftarrow}%
%EndExpansion
\text{ }f\text{ ext. one convex},
\]
follows from Theorem \ref{Thm comme Sverak}.\smallskip
\end{proof}

We also have the following elementary properties.

\begin{proposition}
Let $1\leq k\leq n$ and $f:\Lambda^{k}
\rightarrow\mathbb{R}.$\smallskip

\textbf{(i)} Any ext. one convex function is locally Lipschitz.\smallskip

\textbf{(ii)} If $f$ is ext. one convex and $C^{2},$ for every $\xi\in\Lambda
^{k}$, $\alpha\in\Lambda^{k-1}$ and $\beta\in\Lambda^{1}$,
\[
\sum_{I,J\in\mathcal{T}_{k}^{n}}\frac{\partial
^{2}f\left(  \xi\right)  }{\partial\xi_{I}\partial\xi_{J}}(\alpha\wedge\beta)_{I}(\alpha\wedge\beta)_{J}\geq 0.
\]

\end{proposition}

\begin{proof}
\textbf{(i)} The fact that $f$ is locally Lipschitz follows from the
observation that any ext. one convex function is in fact separately convex.
These last functions are known to be locally Lipschitz (cf. Theorem 2.31 in
\cite{Dacorogna 2007}).\smallskip

\textbf{(ii)} We next assume that $f$ is $C^{2}.$ By definition the function%
\[
g:t\rightarrow g\left(  t\right)  =f\left(  \xi+t\,\alpha\wedge\beta\right)
\]
is convex for every $\xi\in\Lambda^{k},$ $\alpha\in\Lambda^{k-1}$ and
$\beta\in\Lambda^{1}.$ Since $f$ is $C^{2},$ our claim follows from the fact
that $g^{\prime\prime}\left(  0\right)  \geq0.$\smallskip
\end{proof}

We now give an equivalent formulation of ext. quasiconvexity, but before that
we need the following notation.

\begin{notation}
Let $\Omega\subset\mathbb{R}^n$ be a smooth, open set. We define
$$
W_{\delta,T}^{1,\infty}\left(  \Omega;\Lambda^{k}\right)=\left\{\omega\in W^{1,\infty}\left(  \Omega;\Lambda^{k}\right):\delta\omega=0\text{ in }\Omega\text{ and }\nu\wedge\omega=0\text{ on
}\partial\Omega\right\},
$$
where $\nu$ is the outward unit normal to $\partial\Omega$.
\end{notation}

\begin{proposition}
\label{Prop equiv quasiconvexe}Let $f:\Lambda^{k}\rightarrow\mathbb{R}$ be continuous. The following statements
are equivalent.\smallskip

\textbf{(i)} $f$ is \emph{ext. quasiconvex}.\smallskip

\textbf{(ii)} For every bounded smooth open set $\Omega\subset\mathbb{R}^{n}$,
$\psi\in W_{\delta,T}^{1,\infty
}\left(  \Omega;\Lambda^{k-1}\right)  $ and $\xi\in\Lambda^{k}$,
\[
\int_{\Omega}f\left(  \xi+d\psi\right)  \geq f\left(  \xi\right)
\operatorname*{meas}\Omega.
\]

\end{proposition}

\begin{remark}\label{21.6.2013.3}
Given a function $f:\Lambda^{k} \rightarrow
\mathbb{R}$, the \emph{ext. quasiconvex envelope}, which is the largest ext
quasiconvex function below $f,$ is given by (as in Theorem 6.9 of
\cite{Dacorogna 2007})%
\begin{align*}
Q_{ext}f\left(  \xi\right)   &  =\inf\left\{  \frac{1}{\operatorname*{meas}%
\Omega}\int_{\Omega}f\left(  \xi+d\omega\right)  :\omega\in W_{0}^{1,\infty
}\left(  \Omega;\Lambda^{k-1}\right)  \right\}  \smallskip\\
&  =\inf\left\{  \frac{1}{\operatorname*{meas}\Omega}\int_{\Omega}f\left(
\xi+d\psi\right)  :\psi\in W_{\delta,T}^{1,\infty}\left(  \Omega;\Lambda
^{k-1}\right)  \right\}  .
\end{align*}
\end{remark}
\begin{proof}
\textbf{(ii) }$\Rightarrow$\textbf{ (i):} Let $\Omega\subset\mathbb{R}^{n}$ be a bounded smooth open set, $\xi\in\Lambda^{k}$ and let $\omega\in W_{0}^{1,\infty
}\left(  \Omega;\Lambda^{k-1}\right)$. Using density, we find $\omega_\epsilon\in C^\infty_{0}\left(  \Omega;\Lambda^{k-1}\right)$ such that
\begin{equation}\label{20.6.2013.1}
\sup_{\epsilon>0}\Vert\nabla\omega_\epsilon\Vert_{L^\infty}<\infty\text{ and }\omega_\epsilon\rightarrow\omega,\text{ in }W^{1,2
}\left(  \Omega;\Lambda^{k-1}\right).
\end{equation}
Appealing to Theorem 7.2 in
\cite{Csato-Dac-Kneuss 2012}, we now find $\psi_\epsilon\in C^\infty_{\delta,T}\left(
\Omega;\Lambda^{k-1}\right)$ such that%
\[
\left\{
\begin{array}
[c]{cl}%
d\psi_\epsilon=d\omega_\epsilon & \text{in }\Omega\smallskip\\
\delta\psi_\epsilon=0 & \text{in }\Omega\smallskip\\
\nu\wedge\psi_\epsilon=0 & \text{on }\partial\Omega.
\end{array}
\right.
\]
We use (\ref{20.6.2013.1}) to apply
dominated convergence theorem to obtain
\[
\int_{\Omega}f\left(  \xi+d\omega\right)  =\lim_{\epsilon\rightarrow0}\int_{\Omega}f\left(  \xi
+d\omega_{\epsilon}\right)=\lim_{\epsilon\rightarrow0}\int_{\Omega}f\left(  \xi
+d\psi_\epsilon\right)\geq f\left(  \xi\right)  \operatorname*{meas}\Omega,
\]
where we have used \textbf{(ii)} in the last step. Therefore, $f$ is ext. quasiconvex.\smallskip

\textbf{(i) }$\Rightarrow$\textbf{ (ii):} Let $\psi\in W_{\delta,T}^{1,\infty
}\left(  \Omega;\Lambda^{k-1}\right)  .$ Then, by Theorem 8.16 in
\cite{Csato-Dac-Kneuss 2012}, we can find $\omega\in W_{0}^{1,2}\left(
\Omega;\Lambda^{k-1}\right)$ such that%
\[
\left\{
\begin{array}
[c]{cl}%
d\omega=d\psi & \text{in }\Omega\smallskip\\
\omega=0 & \text{on }\partial\Omega.
\end{array}
\right.
\]
With a similar argument as above, we infer that%
\[
\int_{\Omega}f\left(  \xi+d\psi\right)  =\int_{\Omega}f\left(  \xi
+d\omega\right)  =\lim_{\epsilon\rightarrow0}\int_{\Omega}f\left(  \xi
+d\omega_{\epsilon}\right)  \geq f\left(  \xi\right)  \operatorname*{meas}%
\Omega,
\]
as claimed.\smallskip
\end{proof}

We finally have also another formulation of ext. polyconvexity.

\begin{proposition}
\label{Proposition equiv polyconvexite}Let $f:\Lambda^{k}\rightarrow\mathbb{R}.$ The following statements are
then equivalent.\smallskip

\textbf{(i)} The function $f$ is ext. polyconvex.\smallskip

\textbf{(ii)} For every $\xi\in\Lambda^{k},$ there exist $c_{s}=c_{s}\left(
\xi\right)  \in\Lambda^{ks},$ $1\leq s\leq\left[  n/k\right]  ,$ such that%
\[
f\left(  \eta\right)  \geq f\left(  \xi\right)  +\sum_{s=1}^{\left[
n/k\right]  }\left\langle c_{s}\left(  \xi\right)  ;\eta^{s}-\xi
^{s}\right\rangle ,\quad\text{for every }\eta\in\Lambda^{k}.
\]

\textbf{(iii)} Let%
\[
N=\sum_{s=1}^{\left[  n/k\right]}\binom{n}{sk}.
\]
For every $t_{i}\geq0$ with $\sum_{i=1}^{N+1}t_{i}=1$ and every $\xi_{i}%
\in\Lambda^{k}$ such that%
\[
\sum_{i=1}^{N+1}t_{i}\xi_{i}^{s}=\left(  \sum_{i=1}^{N+1}t_{i}\xi_{i}\right)
^{s},\text{\quad for every }1\leq s\leq\left[  n/k\right],
\]
the following inequality holds
\[
f\left(  \sum_{i=1}^{N+1}t_{i}\xi_{i}\right)  \leq\sum_{i=1}^{N+1}%
t_{i}f\left(  \xi_{i}\right)  .
\]

\end{proposition}
\begin{proof}
\textbf{(i) }$\Rightarrow$\textbf{ (ii):} Since $f$ is ext. polyconvex, there
exists a convex function $F$ such that%
\[
f\left(  \xi\right)  =F\left(  \xi,\xi^{2},\ldots,\xi^{\left[  n/k\right]
}\right)  .
\]
$F$ being convex, for every $\xi\in\Lambda^{k}$, there exists $c_{s}%
=c_{s}\left(  \xi\right)  \in\Lambda^{ks},$ $1\leq s\leq\left[  n/k\right]  ,$
such that, for all $\eta\in\Lambda^{k}$,
\[
f\left(  \eta\right)  -f\left(  \xi\right)  =F\left(  \eta,\ldots
,\eta^{\left[  n/k\right]  }\right)  -F\left(  \xi,\ldots,\xi^{\left[
n/k\right]  }\right)  \geq\sum_{s=1}^{\left[  n/k\right]  }\left\langle
c_{s};\eta^{s}-\xi^{s}\right\rangle,
\]
as claimed.\smallskip

\textbf{(ii) }$\Rightarrow$\textbf{ (i):} Conversely assume that the inequality
is valid and, for $\theta=\left(  \theta_{1},\ldots,\theta_{\left[
n/k\right]  }\right)  \in\Lambda^{k}\times\cdots\times\Lambda^{\left[
n/k\right]  k}$, we define
\[
F\left(  \theta\right)  =\sup_{\xi\in\Lambda^{k}}\left\{  f\left(  \xi\right)
+%
%TCIMACRO{\tsum \nolimits_{s=1}^{\left[  n/k\right]  }}%
%BeginExpansion
{\textstyle\sum\nolimits_{s=1}^{\left[  n/k\right]  }}
%EndExpansion
\left\langle c_{s}\left(  \xi\right)  ;\theta_{s}-\xi^{s}\right\rangle
\right\}  .
\]
Clearly $F$ is convex as a supremum of affine functions. Moreover if%
\[
\theta=\left(  \eta,\ldots,\eta^{\left[  n/k\right]  }\right),
\]
then, in view of the inequality, the supremum is attained by $\xi=\eta,$ i.e.%
\[
f\left(  \eta\right)  =F\left(  \eta,\ldots,\eta^{\left[  n/k\right]
}\right),
\]
and thus $f$ is ext. polyconvex.\smallskip

\textbf{(i) }$\Rightarrow$\textbf{ (iii):} Since $f$ is ext. polyconvex, there
exists a convex function $F$ such that%
\[
f\left(  \xi\right)  =F\left(  \xi,\xi^{2},\ldots,\xi^{\left[  n/k\right]
}\right)  .
\]
The convexity of $F$ implies that%
\begin{align*}
f\left(  \sum_{i=1}^{N+1}t_{i}\xi_{i}\right)   &  =F\left(  \left(  \sum
_{i=1}^{N+1}t_{i}\xi_{i}\right)  ,\ldots,\left(  \sum_{i=1}^{N+1}t_{i}\xi
_{i}\right)  ^{\left[  n/k\right]  }\right)  \smallskip\\
&  =F\left(  \sum_{i=1}^{N+1}t_{i}\left(  \xi_{i},\ldots,\xi_{i}^{\left[
n/k\right]  }\right)  \right)  \leq\sum_{i=1}^{N+1}t_{i}F\left(  \xi
_{i},\ldots,\xi_{i}^{\left[  n/k\right]  }\right)  \smallskip\\
&  =\sum_{i=1}^{N+1}t_{i}f\left(  \xi_{i}\right)  .
\end{align*}

\textbf{(iii) }$\Rightarrow$\textbf{ (i):} The proof is based on
Carath\'{e}odory theorem and follows exactly as in Theorem 5.6 in
\cite{Dacorogna 2007}.
\end{proof}

\section{The quasiaffine case}

\subsection{Some preliminary results}

We start with two elementary results.

\begin{lemma}
\label{Lemme elem 1-affine}Let $f:\Lambda^{k}
\rightarrow\mathbb{R}$ be ext. one affine with $1\leq k\leq n.$ Then%
\[
f\left(  \xi+\sum_{i=1}^{N}t_{i}\,\alpha_{i}\wedge a\right)
=f\left(  \xi\right)  +\sum_{i=1}^{N}t_{i}\left[  f\left(  \xi+\alpha
_{i}\wedge a\right)  -f\left(  \xi\right)  \right],
\]
for every $t_{i}\in\mathbb{R},\,\xi\in\Lambda^{k},\,\alpha_{i}\in\Lambda
^{k-1},\,a\in\Lambda^{1}.$
\end{lemma}

\begin{proof}
\emph{Step 1.} Since $f$ is ext. one affine,
\[
f\left(  \xi+t\alpha\wedge a\right)  =f\left(  \xi\right)  +t\left[  f\left(
\xi+\alpha\wedge a\right)  -f\left(  \xi\right)  \right]  .
\]

\emph{Step 2.} Let us first prove that%
\[
f\left(  \xi+\alpha\wedge a+\beta\wedge a\right)  +f\left(  \xi\right)
=f\left(  \xi+\alpha\wedge a\right)  +f\left(  \xi+\beta\wedge a\right)  .
\]
First assume that $s\neq0.$ We have, using Step 1, that%
\begin{align*}
f\left(  \xi+s\,\alpha\wedge a+\beta\wedge a\right)=& f\left(  \xi+s\,\left(  \alpha+\frac{1}{s}\,\beta\right)  \wedge a\right)\\
=& f\left(  \xi\right)  +s\left[  f\left(  \xi+\left(  \alpha+\frac{1}%
{s}\,\beta\right)  \wedge a\right)  -f\left(  \xi\right)  \right],
\end{align*}
and hence, using Step 1 again,%
\begin{align*}
&  f\left(  \xi+s\,\alpha\wedge a+\beta\wedge a\right)  \smallskip\\
&  =f\left(  \xi\right)  +s\left\{  f\left(  \xi+\alpha\wedge a\right)
+\frac{1}{s}\left[  f\left(  \xi+\alpha\wedge a+\beta\wedge a\right)
-f\left(  \xi+\alpha\wedge a\right)  \right]  -f\left(  \xi\right)  \right\}
\smallskip\\
&  =f\left(  \xi\right)  +s\left[  f\left(  \xi+\alpha\wedge a\right)
-f\left(  \xi\right)  \right]  +\left[  f\left(  \xi+\alpha\wedge
a+\beta\wedge a\right)  -f\left(  \xi+\alpha\wedge a\right)  \right]  .
\end{align*}
Since $f$ is continuous, we have the result by letting $s\rightarrow
0.$\smallskip

\emph{Step 3.} We now prove the claim. We proceed by induction. The case $N=1$
is just Step 1. We first use the induction hypothesis to write%
\begin{align*}
&  f\left(  \xi+\sum_{i=1}^{N}t_{i}\,\alpha_{i}\wedge a\right)
=f\left(  \xi+t_{N}\,\alpha_{N}\wedge a+\sum_{i=1}^{N-1}%
t_{i}\,\alpha_{i}\wedge a\right)  \smallskip\\
&  =f\left(  \xi+t_{N}\,\alpha_{N}\wedge a\right)  +\sum_{i=1}^{N-1}%
t_{i}\left[  f\left(  \xi+t_{N}\,\alpha_{N}\wedge a+\alpha_{i}\wedge a\right)
-f\left(  \xi+t_{N}\,\alpha_{N}\wedge a\right)  \right]  .
\end{align*}
We then appeal to Step 1 to get%
\begin{align*}
&  f\left(  \xi+\sum\nolimits_{i=1}^{N}t_{i}\,\alpha_{i}\wedge a\right)
=f\left(  \xi\right)  +t_{N}\left[  f\left(  \xi+\alpha_{N}\wedge a\right)
-f\left(  \xi\right)  \right]  \smallskip\\
&  +\sum_{i=1}^{N-1}t_{i}\left\{
\begin{array}
[c]{c}%
f\left(  \xi+\alpha_{i}\wedge a\right)  +t_{N}\left[  f\left(  \xi+\alpha
_{i}\wedge a+\alpha_{N}\wedge a\right)  -f\left(  \xi+\alpha_{i}\wedge
a\right)  \right] \\
-f\left(  \xi\right)  -t_{N}\left[  f\left(  \xi+\alpha_{N}\wedge a\right)
-f\left(  \xi\right)  \right]
\end{array}
\right\},  \smallskip
\end{align*}
and thus%
\begin{align*}
& f\left(  \xi+\sum\nolimits_{i=1}^{N}t_{i}\,\alpha_{i}\wedge a\right)
=f\left(  \xi\right)  +\sum_{i=1}^{N}t_{i}\left[  f\left(  \xi+\alpha
_{i}\wedge a\right)  -f\left(  \xi\right)  \right]  \smallskip\\
&+t_{N}\sum_{i=1}^{N-1}t_{i}\left\{
f\left(  \xi+\alpha_{i}\wedge a+\alpha_{N}\wedge a\right)  -f\left(
\xi+\alpha_{i}\wedge a\right)
-f\left(  \xi+\alpha_{N}\wedge a\right)  +f\left(  \xi\right)
\right\}.
\end{align*}
Appealing to Step 2, we see that the last term vanishes and therefore the
induction reasoning is complete and this achieves the proof of the
lemma.\smallskip
\end{proof}

We have as an immediate consequence the following result.

\begin{corollary}
\label{Corrolaire sur sommes indep}Let $f:\Lambda^{k}\rightarrow\mathbb{R}$ be ext. one affine with $1\leq k\leq n.$
Then%
\begin{align*}
&  \left[  f\left(  \xi+\alpha\wedge a+\beta\wedge b\right)  -f\left(
\xi\right)  \right]  +\left[  f\left(  \xi+\beta\wedge a+\alpha\wedge
b\right)  -f\left(  \xi\right)  \right]  \smallskip\\
&  =\left[  f\left(  \xi+\alpha\wedge a\right)  -f\left(  \xi\right)  \right]
+\left[  f\left(  \xi+\beta\wedge a\right)  -f\left(  \xi\right)  \right]
\smallskip\\
&  +\left[  f\left(  \xi+\alpha\wedge b\right)  -f\left(  \xi\right)  \right]
+\left[  f\left(  \xi+\beta\wedge b\right)  -f\left(  \xi\right)  \right],
\end{align*}
for every $\xi\in\Lambda^{k},\,\alpha,\beta\in\Lambda^{k-1},\,a,b\in
\Lambda^{1}.$
\end{corollary}

\subsection{The main theorem}

\begin{theorem}
\label{Thm principal quasiaffine}Let $1\leq k\leq n$ and $f:\Lambda^{k}\rightarrow\mathbb{R}.$ The following statements are
then equivalent.\smallskip

\textbf{(i)} $f$ is ext. polyaffine.\smallskip

\textbf{(ii)} $f$ is ext. quasiaffine.\smallskip

\textbf{(iii)} $f$ is ext. one affine.\smallskip

\textbf{(iv)} For every $0\leq s\leq\left[
n/k\right]$, there exist $c_{s}\in\Lambda^{ks}$ such that,
\[
f\left(  \xi\right)  =\sum_{s=0}^{\left[  n/k\right]  }\left\langle c_{s}%
;\xi^{s}\right\rangle,\text{ for every }\xi\in\Lambda^{k}.
\]

\end{theorem}

\begin{remark}
When $k$ is odd (since then $\xi^{s}=0$ for every $s\geq2$) or when $2k>n$ (in
particular when $k=n$ or $k=n-1$), then all the statements are equivalent to
$f$ being affine.
\end{remark}
\begin{proof}
The statements%
\[
(i)\;\Rightarrow\;(ii)\;\Rightarrow\;(iii)
\]
follow at once from Theorem \ref{Thm general sur poly quasi ...}. The
statement%
\[
(iv)\;\Rightarrow\;(i)
\]
is a direct consequence of the definition of ext. polyconvexity. So it only
remains to prove%
\[
(iii)\;\Rightarrow\;(iv).
\]
We divide the proof into three steps.\smallskip

\emph{Step 1.} We first prove that $f$ is a polynomial of degree at
most $n$ and is of the form%
\begin{equation}
f\left(  \xi\right)  =\sum_{s=0}^{n}f_{s}\left(  \xi\right),
\label{(1) Etape 1.1 thm principal}%
\end{equation}
where, for each $s=0,\ldots,n$, $f_{s}$ is a homogeneous polynomial of degree $s$ and is
ext. one affine. To prove (\ref{(1) Etape 1.1 thm principal}), let us
proceed by induction on the dimension $n$. The case $n=1$ is trivial to check. For each $\xi\in\Lambda^k$, we write%
\[
\xi=\sum_{I\in\mathcal{T}^{n}_{k},1\in I }\xi_{I} e^{I} +\xi_{N},\text{ where }\xi_{N}=\sum_{I\in\mathcal{T}^{n}_{k},1\notin I }\xi_{I} e^{I}.
\]
Note that, $\xi_{N}\in \Lambda^k\left(\{e^1\}^{\perp}\right)$.  Invoking Lemma \ref{Lemme elem 1-affine}, we obtain that
\[
f\left(  \xi\right)  =f\left(  \xi_{N}\right)  +\sum_{I\in\mathcal{T}^{n}_{k},1\in I}\xi_{I}\left[  f\left(  \xi_{N}+e^{I}\right)  -f\left(  \xi_{N}\right)
\right].
\]
Since $f$ is ext. one affine on $\Lambda^k\left(\{e^1\}^{\perp}\right)$, applying the induction hypothesis to $f|_{\Lambda^k\left(\{e^1\}^{\perp}\right)}$ and $f\left(e^I+\cdot\right)|_{\Lambda^k\left(\{e^1\}^{\perp}\right)}$, we deduce that both are polynomials of degree at most $\left(  n-1\right)$. Hence, $f$ is a polynomial of degree at most $n$. This proves the claim by induction. That each of $f_{s}$ is ext. one affine, follows from the fact that each $f_{s}$ has a different degree of homogeneity.\smallskip

\emph{Step 2.} We now show that $f$ is, in fact, a polynomial of degree at most $\left[  n/k\right]$, which is equivalent to proving that each $f_s$ in (\ref{(1) Etape 1.1 thm principal}) is a polynomial of degree at most $\left[  n/k\right]$. Since $f_{s}$ is
a homogeneous polynomial of degree $s$, we can write%
\begin{equation}\label{(1) Etape 1.2 thm principal}
f_{s}\left(  \xi\right)  =\sum_{I^{1},\ldots,I^{s}\in \mathcal{T}^{n}_{k}}d_{I^{1}\cdots I^{s}}\,\xi_{I^{1}}\cdots\xi_{I^{s}},
\end{equation}
where $d_{I^{1}\cdots I^{s}}\in\mathbb{R}.$ It is enough to prove that, for some $I^{1},\ldots,I^{s}\in \mathcal{T}^{n}_{k}$, whenever $d_{I^{1}\cdots I^{s}}\neq 0$, we have,
$$
I^p\cap I^q=\emptyset,\text{ for all }p,q=1,\ldots,s;\,p\neq q.
$$
Let us suppose to the contrary that, for some $p,q$, $p\neq q$, we have $
I^p\cap I^q\neq\emptyset.$ For $t\in\mathbb{R}$, let us define
\[
\xi(t)=t\left(e^{I^p}+ e^{I^q}\right)+\sum_{\stackrel{a=1}{a\neq p,q}}^{s}e^{I^{a}}.
\]
We, therefore,
have according to (\ref{(1) Etape 1.2 thm principal}) that, for all $t\in\mathbb{R}$,
\[
f_{s}\left(  \xi(t)\right)  =t^{2}d_{I%
^{1}\cdots I^{s}},\text{ for all }t\in\mathbb{R}.
\]
On the other hand, using Lemma \ref{Lemme elem 1-affine}, it follows that
\begin{align*}
f_{s}\left(  \xi(t)\right)=f_{s}\left( \sum_{\stackrel{a=1}{a\neq p,q}}^{s}e^{I^{a}}+ t\left(e^{I^p}+ e^{I^q}\right)\right)=
f_{s}\left(  \xi(0)\right)+t\left[f_{s}\left(  \xi(1)\right)-f_{s}\left(  \xi(0)\right) \right],
\end{align*}
which is an affine function of $t$. This proves the claim by contradiction.\smallskip

\emph{Step 3.} Henceforth, to avoid any ambiguity, let us fix an order in which multiindices $I^{1},\cdots,I^{s}$ are considered to appear in Equation (\ref{(1) Etape 1.2 thm principal}) and we choose the order to that%
$$
i_{1}^{1}<\cdots<i_{1}^{s},
$$
where $i_{1}^{j}$ is the first element of $I^j$, for all $j=1,\ldots,s$. With the aforementioned order in force, we re-arrange Equation (\ref{(1) Etape 1.2 thm principal}) to have
\begin{equation}
f\left(  \xi\right)  =\sum_{r=0}^{\left[  n/k\right]  }f_{s}\left(
\xi\right),  \quad\text{where}\quad f_{s}\left(  \xi\right)  =\sum_{I^{1},\ldots,I^{s}}c_{I^{1}\cdots I^{s}}\,\xi_{I^{1}}%
\cdots\xi_{I^{s}}, \label{(1) Etape 1 thm principal}%
\end{equation}
with $c_{I^{1}\cdots I^{s}}\in\mathbb{R}\setminus\{0\}$, and the ordered multiindices $I^{1},\ldots,I^s \in\mathcal{T}^{n}_{k}$ with $i_{1}^{1}<\cdots<i_{1}^{s}$.
%From now on we assume that $f$ and $f_{s}$ are as in
%(\ref{(1) Etape 1 thm principal}).
Note that, the theorem is proved once we show that%
\begin{equation*}
f_{s}\left(  \xi\right)  =\left\langle c_{s};\xi^{s}\right\rangle,
\end{equation*}
which is equivalent to proving that,
\begin{equation}\label{(2) Etape 3 thm principal}
c_{J^{1}\cdots J^{s}}=\operatorname*{sgn}%
\left(  \sigma\right)  c_{I^{1}\cdots I^{s}},
\end{equation}
for all $I^{1},\ldots,I^s,J^{1},\ldots,J^s \in\mathcal{T}^{n}_{k}$ satisfying $J^1\cup\cdots\cup J^s = I^1\cup\cdots\cup I^s $, $\sigma \left( J^1J^2\cdots J^s \right) = \left( I^1 I^2\cdots I^s \right)$, where $\sigma \in \mathcal{S}^{sk}$ is the permutation of indices
that respects the aforementioned order. \smallskip

\emph{Step 3.1.} Before concluding the proof, we observe that, for all $s=2,\ldots,n$,
\[
f_{s}\left(  \sum_{i=1}^{s-1}t_{i}\,\alpha_{i}\right)  =0,
\]
where $t_{i}\in\mathbb{R}$ and $\alpha_{i}$ is an element of
the standard basis of $\Lambda^{k}.$ This is a direct consequence of the fact
that $f_{s}$ is homogeneous of degree $s$ and that%
\[
\xi=\sum_{i=1}^{s-1}t_{i}\,\alpha_{i}%
\]
has at most $\left(  s-1\right)  $ coefficients that are non-zero.\smallskip

\emph{Step 3.2.} We finally establish Equation (\ref{(2) Etape 3 thm principal}). Let $I^{1},\ldots,I^s,J^{1},\ldots,J^s \in\mathcal{T}^{n}_{k}$ satisfy $$J^1\cup\cdots\cup J^s = I^1\cup\cdots\cup I^s,\text{ and }\sigma \left( J^1J^2\cdots J^s \right) = \left( I^1 I^2\cdots I^s \right), $$ where
$\sigma \in \mathcal{S}^{sk}$ is the permutation of indices that respects the order. We claim that
\begin{equation}\label{19.6.2013.1}
c_{ J^{1}\cdots J^{s}  }=\operatorname*{sgn}%
\left(  \sigma\right)  c_{I^{1}\cdots I^{s}}.
\end{equation}
Since any permutation that respects the ordering scheme is a product (unique up to parity) of transpositions each of which respects the ordering, it is enough to prove the result for the case where $\sigma$ is a
transposition that respects the ordering. Hence, Equation (\ref{19.6.2013.1}) reduces to proving that
\begin{equation}
c_{I^{1}\cdots I^{s}}=-c_{J^{1} \cdots J^{s} }\,. \label{finalclaim}%
\end{equation}
Let us write
\[
I^{1}=\left(  i_{1}^{1}\,,\ldots,i_{k}^{1}\right)  ,\ldots,I^{s}=\left(  i_{1}^{s}\,,\ldots,i_{k}^{s}\right),
\]
and
\[
J^{1}=\left(  j_{1}^{1}\,,\ldots,j_{k}^{1}\right)  ,\ldots,J^{s}=\left(  j_{1}^{s}\,,\ldots,j_{k}^{s}\right).
\]
Then,
$
i_{1}^{1}<\cdots<i_{1}^{s}\quad\text{and}\quad j_{1}^{1}<\cdots<j_{1}^{s}.
$ Since $\sigma$ respects the ordering scheme, $\sigma$ flips two
indices $i_{r_{1}}^{q_{1}}$ and $i_{r_{2}}^{q_{2}}\,,$ with $q_{1}\neq
q_{2}$ and leaves other indices fixed. Note that, from (\ref{(1) Etape 1 thm principal}), we have%
\begin{equation}
c_{I^{1}\cdots I^{s}}=f_{s}\left(  \sum_{m=1}^{s}e^{I^{m}}\right)
=f_{s}\left(  \sum_{m=1}^{s}e^{i_{1}^{m}}\wedge\cdots\wedge e^{i_{k}^{m}%
}\right),  \label{exprcoeff1}%
\end{equation}
and%
\begin{equation}
c_{J^{1}\cdots J^{s}  }=f_{s}\left(  \sum
_{m=1}^{s}e^{J^{m}  }\right)  =f_{s}\left(
\sum_{m=1}^{s}e^{j_{1}^{m}}\wedge\cdots\wedge e^{j_{k}^{m}}\right)  . \label{exprcoeff2}%
\end{equation}
Since $f_{s}$ is ext. one affine, setting
\[
a=e^{i_{r_{1}}^{q_{1}}},\quad b=e^{i_{r_{2}}^{q_{2}}},\quad\xi=\sum
_{\substack{m=1\\m\neq q_{1}\,,q_{2}}}^{s}e^{i_{1}^{m}}\wedge\cdots\wedge
e^{i_{k}^{m}}=\sum
_{\substack{m=1\\m\neq q_{1}\,,q_{2}}}^{s}
e^{I^{m}},
\]%
\[
\alpha=\pm\,e^{i_{1}^{q_{1}}}\wedge\cdots\wedge\widehat{e^{i_{r_{1}}^{q_{1}}}%
}\wedge\cdots\wedge e^{i_{k}^{q_{1}}}\quad\text{and}\quad\beta=\pm
\,e^{i_{1}^{q_{2}}}\wedge\cdots\wedge\widehat{e^{i_{r_{2}}^{q_{2}}}}%
\wedge\cdots\wedge e^{i_{k}^{q_{2}}},
\]
with signs being chosen appropriately so that
\[
\alpha\wedge a=e^{I^{q_{1}}}=e^{i_{1}^{q_{1}}}\wedge\cdots\wedge
e^{i_{k}^{q_{1}}}\quad\text{and}\quad\beta\wedge b=e^{I^{q_{2}}}%
=e^{i_{1}^{q_{2}}}\wedge\cdots\wedge e^{i_{k}^{q_{2}}},
\]
we can apply Corollary \ref{Corrolaire sur sommes indep} to $f_{s}$ to obtain
\begin{align*}
&  \left[  f_{s}\left(  \xi+\alpha\wedge a+\beta\wedge b\right)  -f_{s}\left(
\xi\right)  \right]  +\left[  f_{s}\left(  \xi+\beta\wedge a+\alpha\wedge
b\right)  -f_{s}\left(  \xi\right)  \right]  \smallskip\\
&  =\left[  f_{s}\left(  \xi+\alpha\wedge a\right)  -f_{s}\left(  \xi\right)
\right]  +\left[  f_{s}\left(  \xi+\beta\wedge b\right)  -f_{s}\left(
\xi\right)  \right]  \smallskip\\
&  +\left[  f_{s}\left(  \xi+\beta\wedge a\right)  -f_{s}\left(  \xi\right)
\right]  +\left[  f_{s}\left(  \xi+\alpha\wedge b\right)  -f_{s}\left(
\xi\right)  \right].
\end{align*}
Using Step 3.1, all except $f_{s}\left(  \xi+\alpha\wedge a+\beta\wedge b\right)$ and
$f_{s}\left(  \xi+\beta\wedge a+\alpha\wedge b\right)$ are $0$. Hence, we deduce
\[
f_{s}\left(  \xi+\alpha\wedge a+\beta\wedge b\right)  =-f_{s}\left(  \xi
+\beta\wedge a+\alpha\wedge b\right),
\]
which, together with (\ref{exprcoeff1}) and (\ref{exprcoeff2}) proves
(\ref{finalclaim}). This concludes the proof of Step 3.2 and thus of the theorem.
\end{proof}

\section{Some examples}

\subsection{The quadratic case}

\subsubsection{Some preliminary results}

Before stating the main theorem on quadratic forms, we need a lemma whose
proof is straightforward.

\begin{lemma}
\label{Lemme formes quadratiques}Let $1\leq k\leq n,$ $M:\Lambda^{k}  \rightarrow\Lambda^{k} $
be a symmetric linear operator and $f:\Lambda^{k}\rightarrow\mathbb{R}$ be such that, for every $\xi\in\Lambda
^{k}  ,$%
\[
f\left(  \xi\right)  =\left\langle M\xi;\xi\right\rangle .
\]
The following statements then hold true.\smallskip

\textbf{(i)} $f$ is ext. polyconvex if and only if there exists $\beta
\in\Lambda^{2k}  $ so that,
\[
f\left(  \xi\right)  \geq\left\langle \beta;\xi\wedge\xi\right\rangle,\text{ for every }\xi
\in\Lambda^{k}.
\]

\textbf{(ii)} $f$ is ext. quasiconvex if and only if%
\[
\int_{\Omega}f\left(  d\omega\right)  \geq0,
\]
for every bounded open set $\Omega\subset\mathbb{R}^n$ and $\omega\in W_{0}^{1,\infty
}\left(  \Omega;\Lambda^{k-1}\right)  .$\smallskip

\textbf{(iii)} $f$ is ext. one convex if and only if, for every $a\in\Lambda^{k-1}$ and $b\in
\Lambda^{1}$,
\[
f\left(  a\wedge b\right)  \geq0,
\]

\end{lemma}

\subsubsection{Some examples}

We start with the following example that will be used in Theorem
\ref{Thm formes quadratiques} below.
\begin{proposition}
\label{Proposition contre-exemple quasi poly (cas quadratique)}Let $2\leq
k\leq n-2.$ Let $\alpha\in\Lambda^{k}  $ be not
$1$-divisible, then there exists $c>0$ such that
\[
f\left(  \xi\right)  =\left\vert \xi\right\vert ^{2}-c\left(  \left\langle
\alpha;\xi\right\rangle \right)  ^{2},%
\]
is ext. quasiconvex but not convex. If, in addition $\alpha\wedge\alpha=0,$
then the above $f,$ for an appropriate $c,$ is ext. quasiconvex but not ext. polyconvex.
\end{proposition}

\begin{remark}
\textbf{(i)} It is easy to see that $\alpha$ is not $1$-divisible if and only
if%
\[
\operatorname*{rank}\left[  \ast\alpha\right]  =n.
\]
This results from Remark 2.44 (iv) (with the help of Proposition 2.33 (iii))
in \cite{Csato-Dac-Kneuss 2012}. Such an $\alpha$ always exists if either of
the following holds (see Propositions 2.37 (ii) and 2.43 in
\cite{Csato-Dac-Kneuss 2012})

- $k=n-2\geq 2$ is even,

- $2\leq k\leq n-3$.

\noindent For example,
\[
\alpha=e^{1}\wedge e^{2}\wedge e^{3}+e^{4}\wedge e^{5}\wedge e^{6}\in
\Lambda^{3}\left(  \mathbb{R}^{6}\right),
\]
is not $1$-divisible.\smallskip

\textbf{(ii)} Note that, when $k=2$ every form $\alpha$ such that $\alpha
\wedge\alpha=0$ is necessarily $1$-divisible. While, as soon as $k$ is even
and $4\leq k\leq n-2,$ there exists $\alpha$ that is not $1$-divisible, but
$\alpha\wedge\alpha=0.$ For example, when $k=4$,
\[
\alpha=e^{1}\wedge e^{2}\wedge e^{3}\wedge e^{4}+e^{1}\wedge e^{2}\wedge
e^{5}\wedge e^{6}+e^{3}\wedge e^{4}\wedge e^{5}\wedge e^{6}\in\Lambda
^{4}\left(  \mathbb{R}^{6}\right)  .
\]
\end{remark}

\begin{proof}
Since the function is quadratic, the notions of ext. one convexity and ext.
quasiconvexity are equivalent (see Theorem \ref{Thm formes quadratiques}). We therefore only need to discuss the ext. one convexity. We divide
the proof into two steps.\smallskip

\emph{Step 1.} We first show that if%
\[
\frac{1}{c}=\sup\left\{  \left(
\left\langle \alpha;a\wedge b\right\rangle \right)  ^{2}:a\in\Lambda^{k-1},\,b\in\Lambda^{1},\,\left\vert a\wedge
b\right\vert =1\right\},
\]
then%
\[
\frac{1}{c}<\left\vert \alpha\right\vert ^{2}.
\]
We prove this statement as follows. Let $a_{s}\in\Lambda^{k-1},\,b_{s}%
\in\Lambda^{1}$ be a maximizing sequence. Up to a subsequence that we do not
relabel, we find that there exists $\lambda\in\Lambda^{k}$ so that%
\[
a_{s}\wedge b_{s}\rightarrow\lambda\quad\text{with}\quad\left\vert
\lambda\right\vert =1.
\]
Similarly, up to a subsequence that we do not relabel, we have that there
exists $\overline{b}\in\Lambda^{1}$ so that%
\[
\frac{b_{s}}{\left\vert b_{s}\right\vert }\rightarrow\overline{b}.
\]
Since%
\[
a_{s}\wedge b_{s}\wedge\frac{b_{s}}{\left\vert b_{s}\right\vert }=0,
\]
we deduce that%
\[
\lambda\wedge\overline{b}=0.
\]
Appealing to Cartan lemma (see Theorem 2.42 in \cite{Csato-Dac-Kneuss 2012}),
we find that there exists $\overline{a}\in\Lambda^{k-1}$ such that%
\[
\lambda=\overline{a}\wedge\overline{b}\quad\text{with}\quad\left\vert
\overline{a}\wedge\overline{b}\right\vert =1.
\]
We therefore have found that%
\[
\frac{1}{c}=\left(  \left\langle \alpha;\overline{a}\wedge\overline
{b}\right\rangle \right)  ^{2}.
\]
Note that $\frac{1}{c}<\left\vert \alpha\right\vert ^{2}$, otherwise
$\overline{a}\wedge\overline{b}$ would be parallel to $\alpha$ and thus
$\alpha$ would be $1$-divisible which contradicts the hypothesis.\smallskip

\emph{Step 2.} So let%
\[
f\left(  \xi\right)  =\left\vert \xi\right\vert ^{2}-c\left(  \left\langle
\alpha;\xi\right\rangle \right)  ^{2}.
\]

(i) Observe that $f$ is not convex since $c\left\vert\alpha\right\vert
^{2}>1$ (by Step 1). Indeed,
\[
f\left(\frac{1}{2}\alpha+\frac{1}{2}(-\alpha)\right)=f(0)=0>\vert\alpha\vert^2\left(1-c\left\vert \alpha\right\vert
^{2}\right)=f(\alpha)=\frac{1}{2}f(\alpha)+\frac{1}{2}f(-\alpha).
\]

(ii) However, $f$ is ext. one convex (and thus, invoking Theorem
\ref{Thm formes quadratiques}, $f$ is ext. quasiconvex). Indeed, let%
\[
g\left(  t\right)  =f\left(  \xi+t\,a\wedge b\right)  =\left\vert
\xi+t\,a\wedge b\right\vert ^{2}-c\left(  \left\langle \alpha;\xi+t\,a\wedge
b\right\rangle \right)  ^{2}.
\]
Note that%
\[
g^{\prime\prime}\left(  t\right)  =2\left[  \left\vert a\wedge b\right\vert
^{2}-c\left(  \left\langle \alpha;a\wedge b\right\rangle \right)  ^{2}\right],
\]
which is non-negative by Step 1. Thus $g$ is convex.\smallskip

(iii) Let $\alpha\wedge\alpha=0$ and assume, for the sake of contradiction,
that $f$ is ext. polyconvex. Then there should exist (cf. Lemma
\ref{Lemme formes quadratiques}) $\beta\in\Lambda^{2k}$ so that, for every
$\xi\in\Lambda^{k},$%
\[
f\left(  \xi\right)  \geq\left\langle \beta;\xi\wedge\xi\right\rangle .
\]
This is clearly impossible, in view of the fact that $c\left\vert
\alpha\right\vert ^{2}>1,$ since choosing $\xi=\alpha,$ we get%
\[
f\left(  \alpha\right)  =\left\vert \alpha\right\vert ^{2}\left(
1-c\left\vert \alpha\right\vert ^{2}\right)  <0=\left\langle \beta
;\alpha\wedge\alpha\right\rangle.
\]

The proof is therefore complete.\smallskip
\end{proof}

We conclude with another example.

\begin{proposition}
Let $1\leq k\leq n,$ $T:\mathbb{R}^{n}\rightarrow\mathbb{R}^{n}$ be a
symmetric linear operator and $T^{\ast}:\Lambda^{k} \rightarrow\Lambda^{k}  $ be the
pullback of $T.$ Let $f:\Lambda^{k}
\rightarrow\mathbb{R}$ be defined as
\[
f\left(  \xi\right)  =\left\langle T^{\ast}\left(  \xi\right)  ;\xi
\right\rangle,\text{ for every }\xi\in\Lambda^{k}.
\]
Then $f$ is ext. one convex if and only if $f$ is convex.
\end{proposition}

\begin{proof}
Since convexity implies ext. one convexity, we only have to prove the reverse
implication.\smallskip

\emph{Step 1.} Since $T$ is symmetric, we can find eigenvalues $\left\{
\lambda_{1},\ldots,\lambda_{n}\right\}$ (not necessarily distinct) of $T$ with a
corresponding set of orthonormal eigenvectors $\left\{  \varepsilon^{1}%
,\ldots,\varepsilon^{n}\right\}  .$ Let $\left\{  e^{1},\ldots,e^{n}\right\}
$ be the standard basis of $\mathbb{R}^{n}$ and let $\Lambda=\operatorname*{diag}%
\left(  \lambda_{1},\ldots,\lambda_{n}\right)  $ and $Q$ be the orthogonal
matrix so that%
\[
Q^{\ast}\left(  \varepsilon^{i}\right)  =e^{i},\quad\text{for }i=1,\ldots,n.
\]
In terms of matrices what we have written just means that%
\[
T=Q\Lambda Q^{t}.
\]
Observe that, for every $i=1,\ldots,n,$%
\begin{align*}
T^{\ast}\left(  \varepsilon^{i}\right)   &  =\left(  Q\Lambda Q^{t}\right)
^{\ast}\left(  \varepsilon^{i}\right)  =\left(  Q^{t}\right)  ^{\ast}\left(
\Lambda^{\ast}\left(  Q^{\ast}\left(  \varepsilon^{i}\right)  \right)
\right)  =\left(  Q^{t}\right)  ^{\ast}\left(  \Lambda^{\ast}\left(
e^{i}\right)  \right)=\lambda_{i}%
\varepsilon^{i}.
\end{align*}
This implies, for every $1\leq k\leq n$ and $I\in\mathcal{T}_{k}^{n}\,,$%
\[
T^{\ast}\left(  \varepsilon^{I}\right)  =T^{\ast}\left(  \varepsilon^{i_{1}%
}\wedge\cdots\wedge\varepsilon^{i_{k}}\right)  =T^{\ast}\left(  \varepsilon
^{i_{1}}\right)  \wedge\cdots\wedge T^{\ast}\left(  \varepsilon^{i_{k}%
}\right)  =\left(  \prod_{j=1}^{k}\lambda_{i_{j}}\right)  \varepsilon^{I}.
\]

\emph{Step 2. }Since $f$ is ext one convex and in view of Lemma
\ref{Lemme formes quadratiques} (iii), we have%
\[
f\left(  \varepsilon^{I}\right)  =\left\langle \left(  T^{\ast}\left(
\varepsilon^{I}\right)  \right)  ;\varepsilon^{I}\right\rangle \geq0
\]
and thus%
\begin{equation}
\prod_{j=1}^{k}\lambda_{i_{j}}=\prod_{i\in I}\lambda_{i}\geq0.
\label{(1) dans Prop rappel}%
\end{equation}
Writing $\xi$ in the basis $\left\{  \varepsilon^{1},\ldots,\varepsilon
^{n}\right\}  ,$ we get%
\begin{align*}
f\left(  \xi\right)   &  =\left\langle T^{\ast}\left(  \xi\right)
;\xi\right\rangle =\left\langle T^{\ast}\left(  \sum_{I\in\mathcal{T}_{k}^{n}%
}\xi_{I}\varepsilon^{I}\right)  ;\sum_{I\in\mathcal{T}_{k}^{n}}\xi
_{I}\varepsilon^{I}\right\rangle=\sum_{I\in\mathcal{T}_{k}^{n}}\left(  \prod_{i\in I}\lambda_{i}\right)
\left(  \xi_{I}\right)  ^{2},
\end{align*}
which according to (\ref{(1) dans Prop rappel}) is non negative. This shows
that $f$ is convex as wished.
\end{proof}

\subsubsection{The main result}

We now turn to the main theorem.

\begin{theorem}
\label{Thm formes quadratiques}Let $1\leq k\leq n,$ $M:\Lambda^{k}  \rightarrow\Lambda^{k} $
be a symmetric linear operator and $f:\Lambda^{k}\rightarrow\mathbb{R}$ be such that, for every $\xi\in\Lambda
^{k}$,%
\[
f\left(  \xi\right)  =\left\langle M\xi;\xi\right\rangle .
\]

\textbf{(i)} The following equivalence holds in all cases%
\[
f\text{ ext. quasiconvex }\Leftrightarrow\text{ }f\text{ ext. one convex.}%
\]

\textbf{(ii)} Let $k=2.$ If $n=2$ or $n=3,$ then%
\[
f\text{ convex }\Leftrightarrow\text{ }f\text{ ext. polyconvex }%
\Leftrightarrow\text{ }f\text{ ext. quasiconvex }\Leftrightarrow\text{
}f\text{ ext. one convex.}%
\]
If $n=4,$ then%
\[
f\text{ convex }%
%TCIMACRO{\QATOP{\Rightarrow}{\nLeftarrow}}%
%BeginExpansion
\genfrac{}{}{0pt}{}{\Rightarrow}{\nLeftarrow}%
%EndExpansion
\text{ }f\text{ ext. polyconvex }\Leftrightarrow\text{ }f\text{ ext.
quasiconvex }\Leftrightarrow\text{ }f\text{ ext. one convex},
\]
while if $n\geq6,$ then%
\[
f\text{ ext. polyconvex }%
%TCIMACRO{\QATOP{\Rightarrow}{\nLeftarrow}}%
%BeginExpansion
\genfrac{}{}{0pt}{}{\Rightarrow}{\nLeftarrow}%
%EndExpansion
\text{ }f\text{ ext. quasiconvex }\Leftrightarrow\text{ }f\text{ ext. one
convex.}%
\]

\textbf{(iii)} If $k$ is odd or if $2k>n,$ then%
\[
f\text{ convex }\Leftrightarrow\text{ }f\text{ ext. polyconvex.}%
\]

\textbf{(iv)} If $k$ is even and $2k\leq n,$ then%
\[
f\text{ convex }%
%TCIMACRO{\QATOP{\Rightarrow}{\nLeftarrow}}%
%BeginExpansion
\genfrac{}{}{0pt}{}{\Rightarrow}{\nLeftarrow}%
%EndExpansion
\text{ }f\text{ ext. polyconvex.}%
\]

\textbf{(v)} If either $3\leq k\leq n-3$ or $k=n-2\geq4$ is even, then%
\[
f\text{ ext. polyconvex }%
%TCIMACRO{\QATOP{\Rightarrow}{\nLeftarrow}}%
%BeginExpansion
\genfrac{}{}{0pt}{}{\Rightarrow}{\nLeftarrow}%
%EndExpansion
\text{ }f\text{ ext. quasiconvex }\Leftrightarrow\text{ }f\text{ ext. one
convex.}%
\]

\end{theorem}

\begin{remark}
\textbf{(i)} We recall that when $k=1$ all notions of convexity are
equivalent.\smallskip

\textbf{(ii)} When $k=2$ and $n=5,$ the equivalence between ext. polyconvexity and
ext. quasiconvexity remains open.
\end{remark}

\begin{proof}
\textbf{(i)} The result follows from Lemma \ref{Lemme formes quadratiques} and Plancherel formula. The proof is similar to that of the case of the gradient (cf. Theorem 5.25 and Lemma 5.28 in \cite{Dacorogna 2007}).\smallskip

\textbf{(ii)} If $n=2$ or $n=3,$ the result follows from Theorem
\ref{Thm general sur poly quasi ...} (ii). If $n\geq6,$ see Theorem
\ref{Thm contre exemple quadratique k=2}. So we now assume that $n=4$ (for the
counter implication see (iv) below). We only have to prove that%
\[
f\text{ ext. one convex }\Rightarrow\text{ }f\text{ ext. polyconvex.}%
\]
We know (by ext. one convexity) that, for every $a,b\in\Lambda^{1}\left(
\mathbb{R}^{4}\right)  $%
\[
f\left(  a\wedge b\right)  \geq0,
\]
and we wish to show (cf. Lemma \ref{Lemme formes quadratiques}) that we can
find $\alpha\in\Lambda^{4}\left(  \mathbb{R}^{4}\right)  $ so that%
\[
f\left(  \xi\right)  \geq\left\langle \alpha;\xi\wedge\xi\right\rangle .
\]

\emph{Step 1. }Let us change slightly the notations and write $\xi\in
\Lambda^{2}\left(  \mathbb{R}^{4}\right)  $ as a vector of $\mathbb{R}^{6}$ in
the following manner%
\[
\xi=\left(  \xi_{12},\xi_{13},\xi_{14},\xi_{23},\xi_{24},\xi_{34}\right),
\]
and therefore $f$ can be seen as a quadratic form over $\mathbb{R}^{6}$ which
is non-negative whenever the quadratic form (note also that $g$ is indefinite)%
\[
g\left(  \xi\right)  =\left\langle e^{1}\wedge e^{2}\wedge e^{3}\wedge
e^{4};\xi\wedge\xi\right\rangle =2\left(  \xi_{12}\xi_{34}-\xi_{13}\xi
_{24}+\xi_{14}\xi_{23}\right)
\]
vanishes. Indeed note that, by Proposition 2.37 in \cite{Csato-Dac-Kneuss 2012},
\[
g\left(  \xi\right)  =0\text{ }\Leftrightarrow\text{ }\xi\wedge\xi=0\text{
}\Leftrightarrow\text{ }\operatorname*{rank}\left[  \xi\right]\in\{0,2\}.
\]
By Proposition 2.43 in \cite{Csato-Dac-Kneuss 2012}, this last condition is equivalent to the existence of $a,b\in\Lambda
^{1}\left(  \mathbb{R}^{4}\right)  $ so that%
\[
\xi=a\wedge b,
\]
and by ext. one convexity we know that $f\left(  a\wedge b\right)  \geq
0.$\smallskip

\emph{Step 2. }We now invoke Theorem 2 in \cite{Marcellini 1984} to find
$\lambda\in\mathbb{R}$ such that%
\[
f\left(  \xi\right)  -\lambda g\left(  \xi\right)  \geq0.
\]
But this is exactly what we had to prove.\smallskip

\textbf{(iii)} This is a general fact. (See Remark \ref{Remarque sur les def de convexite}\,(ii) and Theorem
\ref{Thm general sur poly quasi ...})\smallskip

\textbf{(iv)} The counterexample is just%
\[
f\left(  \xi\right)  =\left\langle \alpha;\xi\wedge\xi\right\rangle,
\]
for any $\alpha\in\Lambda^{2k}$, $\alpha\neq
0.$\smallskip

\textbf{(v)} This is just Proposition
\ref{Proposition contre-exemple quasi poly (cas quadratique)} and the remark
following it. Indeed we consider the two following cases.

- If $k$ is odd (and since $3\leq k\leq n-3,$ then $n\geq6$), we know
from (iii) that $f$ is ext. polyconvex if and only if $f$ is convex and we
also know that there exists an exterior $k$-form which is not $1$-divisible.
Proposition \ref{Proposition contre-exemple quasi poly (cas quadratique)}
gives therefore the result.

- If $k$ is even and $4\leq k\leq n-2$ (which implies again $n\geq6$), then
there exists an exterior $k$-form $\alpha$ which is not $1$-divisible, but $\alpha
\wedge\alpha=0.$ The result thus follows again by Proposition
\ref{Proposition contre-exemple quasi poly (cas quadratique)}.
\end{proof}

\subsubsection{A counterexample for $k=2$}

We now turn to a counterexample that has been mentioned in Theorem
\ref{Thm formes quadratiques}.

\begin{theorem}
\label{Thm contre exemple quadratique k=2}Let $n\geq6.$ Then there exists a
quadratic form $f:\Lambda^{2}  \rightarrow
\mathbb{R}$ which is ext. one convex but not ext. polyconvex.
\end{theorem}

\begin{proof}
It is enough to establish the theorem for $n=6.$ Our counterexample is inspired by Serre
\cite{Serre 1983} and Terpstra \cite{Terpstra} (see Theorem 5.25 (iii) in
\cite{Dacorogna 2007}). It is more convenient to write here $\xi\in\Lambda
^{2}\left(  \mathbb{R}^{6}\right)  $ as%
\[
\xi=\sum_{1\leq i<j\leq6}\xi_{j}^{i}\,e^{i}\wedge e^{j}.
\]
So let%
\[
g\left(  \xi\right)  =\left(  \xi_{2}^{1}\right)  ^{2}+\left(  \xi_{3}%
^{1}\right)  ^{2}+\left(  \xi_{3}^{2}\right)  ^{2}+\left(  \xi_{5}^{4}\right)
^{2}+\left(  \xi_{6}^{4}\right)  ^{2}+\left(  \xi_{6}^{5}\right)
^{2}+h\left(  \xi\right),
\]
where%
\[
h\left(  \xi\right)  =\left(  \xi_{4}^{1}-\xi_{5}^{3}-\xi_{6}^{2}\right)
^{2}+\left(  \xi_{5}^{1}-\xi_{4}^{3}+\xi_{6}^{1}\right)  ^{2}+\left(  \xi
_{4}^{2}-\xi_{4}^{3}-\xi_{6}^{1}\right)  ^{2}+\left(  \xi_{5}^{2}\right)
^{2}+\left(  \xi_{6}^{3}\right)  ^{2}.
\]
Note that $g\geq0.$ We claim that there exists $\gamma>0$ so that%
\[
f\left(  \xi\right)  =g\left(  \xi\right)  -\gamma\left\vert \xi\right\vert
^{2}%
\]
is ext. one convex (cf. Step 1) but not ext. polyconvex (cf. Step 2).\smallskip

\emph{Step 1.} Define%
\begin{equation}\label{21.6.2013.1}
\gamma:=\inf\left\{  g\left(  a\wedge b\right)  :a,b\in\Lambda^{1}\left(
\mathbb{R}^{6}\right)  ,\;\left\vert a\wedge b\right\vert =1\right\}  .
\end{equation}
Note that $\gamma\geq 0$, and it follows from Lemma \ref{Lemme formes quadratiques} that $f$ is ext. one convex.\smallskip

\noindent Let us claim that, in fact $\gamma>0$, which will imply in \emph{Step 2} that $f$ is not ext. polyconvex. To show that $\gamma>0$, we proceed by contradiction and assume that $\gamma=0.$ This implies that we
can find $a,b\in\Lambda^{1}\left(  \mathbb{R}^{6}\right)  $ with $\left\vert
a\wedge b\right\vert =1$ such that%
\[
\left\{
\begin{array}
[c]{c}%
a^{1}b_{2}-a^{2}b_{1}=0\\
a^{1}b_{3}-a^{3}b_{1}=0\\
a^{2}b_{3}-a^{3}b_{2}=0
\end{array}
\right.  \quad\left\{
\begin{array}
[c]{c}%
a^{4}b_{5}-a^{5}b_{4}=0\\
a^{4}b_{6}-a^{6}b_{4}=0\\
a^{5}b_{6}-a^{6}b_{5}=0
\end{array}
\right.  \quad\left\{
\begin{array}
[c]{c}%
a^{2}b_{5}-a^{5}b_{2}=0\\
a^{3}b_{6}-a^{6}b_{3}=0
\end{array}
\right.
\]%
\[
\left\{
\begin{array}
[c]{c}%
\left(  a^{1}b_{4}-a^{4}b_{1}\right)  -\left(  a^{3}b_{5}-a^{5}b_{3}\right)
-\left(  a^{2}b_{6}-a^{6}b_{2}\right)  =0\\
\left(  a^{1}b_{5}-a^{5}b_{1}\right)  -\left(  a^{3}b_{4}-a^{4}b_{3}\right)
+\left(  a^{1}b_{6}-a^{6}b_{1}\right)  =0\\
\left(  a^{2}b_{4}-a^{4}b_{2}\right)  -\left(  a^{3}b_{4}-a^{4}b_{3}\right)
-\left(  a^{1}b_{6}-a^{6}b_{1}\right)  =0.
\end{array}
\right.
\]
Let us introduce some notation, we write%
\[
\underline{a}=\left(
\begin{array}
[c]{c}%
a^{1}\\
a^{2}\\
a^{3}%
\end{array}
\right)  ,\quad\underline{b}=\left(
\begin{array}
[c]{c}%
b_{1}\\
b_{2}\\
b_{3}%
\end{array}
\right)  ,\quad\overline{a}=\left(
\begin{array}
[c]{c}%
a^{4}\\
a^{5}\\
a^{6}%
\end{array}
\right)  ,\quad\overline{b}=\left(
\begin{array}
[c]{c}%
b_{4}\\
b_{5}\\
b_{6}%
\end{array}
\right)  .
\]
Note that the first and second sets of equations lead to%
\[
\underline{a}\Vert\underline{b}\text{\quad and\quad}\overline{a}\Vert
\overline{b}.
\]
We consider two cases starting with the generic case.\smallskip

\emph{Case 1:} there exist $\lambda,\mu\in\mathbb{R}$ such that%
\[
\underline{a}=\lambda\,\underline{b}\text{\quad and\quad}\overline{a}%
=\mu\,\overline{b}.
\]
(The same reasoning applies to all other cases, for example $\underline{b}=\lambda\,\underline{a}$
and $\overline{b}=\mu\,\overline{a}$). Note that $\lambda\neq\mu,$ otherwise
we would have $a=\lambda\,b$ and thus $a\wedge b=0$ contradicting the fact
that $\left\vert a\wedge b\right\vert =1.$ Inserting this in the third and
fourth sets of equations we get%
\[
\left\{
\begin{array}
[c]{c}%
\left(  \lambda-\mu\right)  b_{2}b_{5}=0\\
\left(  \lambda-\mu\right)  b_{3}b_{6}=0
\end{array}
\right.  \quad\left\{
\begin{array}
[c]{c}%
\left(  \lambda-\mu\right)  \left[  b_{1}b_{4}-b_{3}b_{5}-b_{2}b_{6}\right]
=0\\
\left(  \lambda-\mu\right)  \left[  b_{1}b_{5}-b_{3}b_{4}+b_{1}b_{6}\right]
=0\\
\left(  \lambda-\mu\right)  \left[  b_{2}b_{4}-b_{3}b_{4}-b_{1}b_{6}\right]
=0
\end{array}
\right.
\]
and thus, since $\lambda\neq\mu,$%
$$
b_{2}b_{5}=b_{3}b_{6}=b_{1}b_{4}-b_{3}b_{5}-b_{2}b_{6}=
b_{1}b_{5}-b_{3}b_{4}+b_{1}b_{6}=
b_{2}b_{4}-b_{3}b_{4}-b_{1}b_{6}=0.
$$

We have to consider separately the cases $b_{2}=b_{3}=0,$ $b_{5}=b_{6}=0,$
$b_{2}=b_{6}=0$ and $b_{3}=b_{5}=0.$ We do only the first case. Other cases can be handled similarly. Let us assume that $b_{2}=b_{3}=0.$
We thus have%
$$
b_{1}b_{4}=b_{1}b_{5}+b_{1}b_{6}=b_{1}b_{6}=0.
$$
So either $b_{1}=0$ and thus $\underline{b}=0$ and hence $\underline{a}=0$ and
again this implies that $a=\mu\,b$ which contradicts the fact that $\left\vert
a\wedge b\right\vert =1.$ Or $b_{4}=b_{5}=b_{6}=0$ and thus $\overline
{b}=\overline{a}=0$ which as before contradicts the fact that $\left\vert
a\wedge b\right\vert =1.$

\smallskip

\emph{Case 2:} $\underline{b}=0$ and $\overline{a}=0$ (or $\underline{a}=0$
and $\overline{b}=0$ which is handled similarly). This means that $a^{4}%
=a^{5}=a^{6}=0$ and $b_{1}=b_{2}=b_{3}=0.$ We therefore have%
$$
a^{2}b_{5}=a^{3}b_{6}=a^{1}b_{4}-a^{3}b_{5}-a^{2}b_{6}=a^{1}b_{5}-a^{3}b_{4}+a^{1}b_{6}=a^{2}b_{4}-a^{3}b_{4}-a^{1}b_{6}=0.
$$
Four cases can happen $a^{2}=a^{3}=0,$ $a^{2}=b_{6}=0,$ $a^{3}=b_{5}=0$ and
$b_{5}=b_{6}=0.$ Again, we handle only the first case. Other cases can be treated similarly.
Assuming that $a^{2}=a^{3}=0$, we have
\[
a^{1}b_{4}=a^{1}b_{5}+a^{1}b_{6}=a^{1}b_{6}=0.
\]
So either $a^{1}=0$ and thus $a=0$ which is impossible. Or $b_{4}=b_{5}%
=b_{6}=0$ and thus $b=0$ which again cannot happen. Hence, we have proved that $\gamma$ defined in Equation (\ref{21.6.2013.1}) is positive.\smallskip

\emph{Step 2.} We now show that $f$ is not ext. polyconvex. In view of Lemma
\ref{Lemme formes quadratiques} (i), it is sufficient to show that for every
$\alpha\in\Lambda^{4}\left(  \mathbb{R}^{6}\right)  ,$ there exists $\xi
\in\Lambda^{2}\left(  \mathbb{R}^{6}\right)  $ such that%
\[
f\left(  \xi\right)  +\frac{1}{2}\left\langle \alpha;\xi\wedge\xi\right\rangle
<0.
\]
We prove that the above inequality holds for forms $\xi$ of the following form%
\begin{align*}
\xi &  =\left(  b+d\right)  e^{1}\wedge e^{4}+\left(  c-a\right)  e^{1}\wedge
e^{5}+a\, e^{1}\wedge e^{6}\smallskip\\
&  +\left(  c+a\right)  e^{2}\wedge e^{4}+b\,  e^{2}\wedge
e^{6}+c\,  e^{3}\wedge e^{4}+d\,  e^{3}\wedge e^{5}.
\end{align*}
Note that%
\begin{align*}
\frac{1}{2}\xi\wedge\xi &  =\left(  c^{2}-a^{2}\right)  e^{1}\wedge
e^{2}\wedge e^{4}\wedge e^{5}+\left(  ac+a^{2}-b^{2}-bd\right)  e^{1}\wedge
e^{2}\wedge e^{4}\wedge e^{6}\smallskip\\
&  +\left(  ab-bc\right)  e^{1}\wedge e^{2}\wedge e^{5}\wedge e^{6}+\left(
c^{2}-ac-bd-d^{2}\right)  e^{1}\wedge e^{3}\wedge e^{4}\wedge e^{5}%
\smallskip\\
&  +ac\,  e^{1}\wedge e^{3}\wedge e^{4}\wedge e^{6}+ad\, e^{1}\wedge e^{3}\wedge e^{5}\wedge e^{6}\smallskip\\
&  +\left(  -cd-ad\right)  e^{2}\wedge e^{3}\wedge e^{4}\wedge e^{5}+bc\,  e^{2}\wedge e^{3}\wedge e^{4}\wedge e^{6}\smallskip\\
&  +bd\, e^{2}\wedge e^{3}\wedge e^{5}\wedge e^{6}%
\end{align*}
For such forms we have $g\left(  \xi\right)  =0$ and therefore%
\[
f\left(  \xi\right)  =-\gamma\left\vert \xi\right\vert ^{2}=-\gamma\left[
\left(  b+d\right)  ^{2}+\left(  c-a\right)  ^{2}+a^{2}+\left(  c+a\right)
^{2}+b^{2}+c^{2}+d^{2}\right].
\]
Moreover,
\begin{align*}
\frac{1}{2}\left\langle \alpha;\xi\wedge\xi\right\rangle  &  =\alpha
_{1245}\left(  c^{2}-a^{2}\right)  +\alpha_{1246}\left(  ac+a^{2}%
-b^{2}-bd\right)  \smallskip\\
&  +\alpha_{1256}\left(  ab-bc\right)  +\alpha_{1345}\left(  c^{2}%
-ac-bd-d^{2}\right)  +\alpha_{1346}\,ac \smallskip\\
&  +\alpha_{1356}\,ad  +\alpha_{2345}\left(  -cd-ad\right)
+\alpha_{2346}\,bc  +\alpha_{2356}\,bd  .
\end{align*}
We consider three cases.\smallskip

\textit{Case 1}. If $\alpha_{1246}>0,$ then take $a=c=d=0$ and $b\neq0,$ to
get%
$$
f\left(  \xi\right)  +\frac{1}{2}\left\langle \alpha;\xi\wedge\xi
\right\rangle =-\gamma\left(  2b^{2}\right)  -\alpha_{1246}b^{2}<0.
$$

\textit{Case 2}. If $\alpha_{1345}>0,$ then take $a=b=c=0$ and $d\neq0,$ to
get%
\[
f\left(  \xi\right)  +\frac{1}{2}\left\langle \alpha;\xi\wedge\xi\right\rangle
=-\gamma\left(  2d^{2}\right)  -\alpha_{1345}d^{2}<0.
\]
We therefore can assume that $\alpha_{1246}\leq0$ and $\alpha_{1345}\leq
0.$\smallskip

\textit{Case 3}. If $\alpha_{1245}+\alpha_{1345}<0$ $\left(  \alpha_{1246}%
\leq0,\text{ }\alpha_{1345}\leq0\right)  ,$ then take $a=b=d=0$ and $c\neq0$
to get%
\[
f\left(  \xi\right)  +\frac{1}{2}\left\langle \alpha;\xi\wedge\xi\right\rangle
=-\gamma\left(  3c^{2}\right)  +\left(  \alpha_{1245}+\alpha_{1345}\right)
c^{2}<0.
\]
We therefore assume $\alpha_{1246}\leq0,$ $\alpha_{1345}\leq0$ and
$\alpha_{1245}+\alpha_{1345}\geq0.$ From these three inequalities we deduce
that $\alpha_{1246}-\alpha_{1245}\leq0,$ and then taking $b=c=d=0$ and
$a\neq0,$ we get%
\[
f\left(  \xi\right)  +\frac{1}{2}\left\langle \alpha;\xi\wedge\xi\right\rangle
=-\gamma\left(  3a^{2}\right)  +\left(  \alpha_{1246}-\alpha_{1245}\right)
a^{2}<0.
\]
This concludes the proof of the theorem.$\smallskip$
\end{proof}

\subsection{Ext. one convexity does not imply ext. quasiconvexity}

We now give an important counterexample for any $k\geq2.$ It is an adaptation
of the fundamental result of Sverak \cite{Sverak 1992a} (see also Theorem 5.50
in \cite{Dacorogna 2007}).

\begin{theorem}
\label{Thm comme Sverak}Let $2\leq k\leq n-3.$ Then there exists
$f:\Lambda^{k}  \rightarrow\mathbb{R}$ ext. one
convex but not ext. quasiconvex.
\end{theorem}

\begin{remark}
We know that when $k=1,n-1,n$ or $k=n-2$ is odd, then%
\[
f\text{ convex }\Leftrightarrow\text{ }f\text{ ext. polyconvex }%
\Leftrightarrow\text{ }f\text{ ext. quasiconvex }\Leftrightarrow\text{
}f\text{ ext. one convex.}%
\]
Therefore only the case $k=n-2\geq2$ even (including $k=2$ and $n=4$) remains open.
\end{remark}

The main algebraic tool in order to adapt Sverak example is given in the
following lemma.

\begin{lemma}
\label{Lemme algebrique Sverak}Let $k\geq2$ and $n=k+3.$ There exist%
\[
\alpha,\beta,\gamma\in\operatorname*{span}\left\{  e^{i_{1}}\wedge\cdots\wedge
e^{i_{k-1}},\quad3\leq i_{1}<\cdots<i_{k-1}\leq k+3\right\}  \subset
\Lambda^{k-1}\left(  \mathbb{R}^{k+3}\right),
\]
such that if%
\[
L=\operatorname*{span}\left\{  e^{1}\wedge\alpha,e^{2}\wedge\beta,\left(
e^{1}+e^{2}\right)  \wedge\gamma\right\},
\]
then any $1$-divisible $\xi=(x,y,z)=x e^{1}\wedge\alpha + ye^{2}\wedge\beta +z\left(
e^{1}+e^{2}\right)  \wedge\gamma \in L$ necessarily verifies%
\[
xy=xz=yz=0.
\]

\end{lemma}
\begin{proof}
\emph{Step 1.} We choose, recall that $n=k+3,$%
\[
\alpha=\left\{
\begin{array}
[c]{cl}%
\sum_{i=2}^{l+1}\left(  \widehat{e^{2i}}\wedge\widehat{e^{2i+1}}\right)  &
\text{if }k=2l,\smallskip\\
\sum_{i=2}^{l+2}\left(  \widehat{e^{2i-1}}\wedge\widehat{e^{2i}}\right)  &
\text{if }k=2l+1.
\end{array}
\right.
\]%
\[
\beta=\left\{
\begin{array}
[c]{cl}%
\widehat{e^{3}}\wedge\widehat{e^{2l+3}} & \text{if }k=2l,\smallskip\\
\left(  \widehat{e^{3}}\wedge\widehat{e^{5}}\right)  +\left(  \widehat{e^{4}%
}\wedge\widehat{e^{6}}\right)  & \text{if }k=3,\smallskip\\
\sum_{i=2}^{l}\left(  \widehat{e^{2i-1}}\wedge\widehat{e^{2i}}\right)  &
\text{if }k=2l+1\text{ and }k\geq5.
\end{array}
\right.
\]%
\[
\gamma=\left\{
\begin{array}
[c]{cl}%
\sum_{i=2}^{l+1}\left(  \widehat{e^{2i-1}}\wedge\widehat{e^{2i}}\right)  &
\text{if }k=2l,\smallskip\\
\left(  \widehat{e^{2l+1}}\wedge\widehat{e^{2l+4}}\right)  +\left(
\widehat{e^{2l+2}}\wedge\widehat{e^{2l+3}}\right)  & \text{if }k=2l+1,
\end{array}
\right.
\]
where we write, by abuse of notations, for $3\leq i< j\leq k+3$,
\[
\widehat{e^{i}}\wedge\widehat{e^{j}}=e^{3}\wedge\cdots\wedge\widehat{e^{i}%
}\wedge\cdots\wedge\widehat{e^{j}}\wedge\cdots\wedge e^{k+3}.
\]
Observe that $\left\{  \alpha,\beta,\gamma\right\}  $ are linearly
independent.\smallskip

\emph{Step 2.} We now prove the statement, namely that if $\xi=\left(
x,y,z\right)  \in L$ is $1$-divisible (i.e. $\xi=b\wedge a$ for $a\in
\Lambda^{1}$ and $b\in\Lambda^{k-1}$), then necessarily%
\[
xy=xz=yz=0.
\]
Assume that $\xi\neq0$ (otherwise the result is trivial) and thus $a\neq0.$
Note that if $\xi=b\wedge a,$ then $a\wedge\xi=0.$ We write%
\[
a=\sum_{i=1}^{k+3}a_{i}\,e^{i}\neq0.
\]

\emph{Step 2.1.} Since $a\wedge\xi=0$ we deduce that the term involving
$e^{1}\wedge e^{2}$ must be $0$ and thus%
\[
-a_{2}x\,\alpha+a_{1}y\,\beta+\left(  a_{1}-a_{2}\right)  z\,\gamma=0.
\]
Since $\left\{  \alpha,\beta,\gamma\right\}  $ are linearly independent, we
deduce that%
\[
a_{2}x=a_{1}y=\left(  a_{1}-a_{2}\right)  z=0.
\]
From there we infer that $xy=xz=yz=0,$ as soon as either $a_{1}\neq0$ or
$a_{2}\neq0.$ So in order to establish the lemma it is enough to consider $a$
of the form%
\[
a=\sum_{i=3}^{k+3}a_{i}\,e^{i}\neq0.
\]
We therefore have%
\[
\sum_{i=3}^{k+3}a_{i}\,e^{i}\wedge\left[  e^{1}\wedge\left(  x\,\alpha
+z\,\gamma\right)  +e^{2}\wedge\left(  y\,\beta+z\,\gamma\right)  \right]  =0,
\]
which implies that%
\begin{equation}
\left\{
\begin{array}
[c]{c}%
a\wedge\left(  x\,\alpha+z\,\gamma\right)  =\sum_{i=3}^{k+3}a_{i}\,e^{i}%
\wedge\left(  x\,\alpha+z\,\gamma\right)  =0,\smallskip\\
a\wedge\left(  y\,\beta+z\,\gamma\right)  =\sum_{i=3}^{k+3}a_{i}\,e^{i}%
\wedge\left(  y\,\beta+z\,\gamma\right)  =0.
\end{array}
\right.  \label{(1) dans Contre ex sverak}%
\end{equation}
We continue the discussion considering separately the cases $k$ even, $k=3$
and $k\geq5$ odd. They are all treated in the same way and we prove it only in
the even case.\smallskip

\emph{Step 2.2: }$k=2l\geq2.$ We have to prove that if%
\[
a=\sum_{i=3}^{2l+3}a_{i}e^{i}\neq0
\]
satisfies (\ref{(1) dans Contre ex sverak}), then necessarily%
\[
xy=xz=yz=0.
\]
We find (up to a $+$ or $-$ sign but here it is immaterial)%
\begin{align*}
a\wedge\alpha=&\sum_{i=2}^{l+1}\left(  a_{2i+1}\widehat{e^{2i}}\right)
+\sum_{i=2}^{l}\left(  a_{2i}\widehat{e^{2i+1}}\right)  +a_{2l+2}%
\widehat{e^{2l+3}},\\
a\wedge\beta=&a_{2l+3}\widehat{e^{3}}+a_{3}\widehat{e^{2l+3}},\\
a\wedge\gamma=&a_{4}\widehat{e^{3}}+\sum_{i=2}^{l+1}\left(  a_{2i-1}%
\widehat{e^{2i}}\right)  +\sum_{i=2}^{l}\left(  a_{2i+2}\widehat{e^{2i+1}%
}\right)  .
\end{align*}
Therefore,
\begin{align*}
a\wedge\left(  x\,\alpha+z\,\gamma\right)  =&z\,a_{4}\widehat{e^{3}}+\sum
_{i=2}^{l+1}\left(  x\,a_{2i+1}+z\,a_{2i-1}\right)  \widehat{e^{2i}}%
\\&+\sum_{i=2}^{l}\left(  x\,a_{2i}+z\,a_{2i+2}\right)  \widehat{e^{2i+1}%
}+x\,a_{2l+2}\widehat{e^{2l+3}},\\
a\wedge\left(  y\,\beta+z\,\gamma\right)  =&\left(  y\,a_{2l+3}+z\,a_{4}%
\right)  \widehat{e^{3}}+z\left\{  \sum_{i=2}^{l+1}\left(  a_{2i-1}%
\widehat{e^{2i}}\right)  +\sum_{i=2}^{l}\left(  a_{2i+2}\widehat{e^{2i+1}%
}\right)  \right\}\\&  +y\,a_{3}\widehat{e^{2l+3}}.
\end{align*}

\emph{Case 1 :} $x=z=0.$ This is our claim.\smallskip

\emph{Case 2 :} $z=0$ and $x\neq0.$ We can also assume that $y\neq0$ otherwise
we have the claim $y=z=0.$ From the first equation we obtain%
\[
a_{2i}=a_{2i+1}=0,\quad i=2,\cdots,l+1.
\]
So only $a_{3}$ might be non-zero. However since $y\neq0$ we deduce from the
second equation that $a_{3}=0$ and thus $a=0$ which is impossible.\smallskip

\emph{Case 3 :} $x=0$ and $z\neq0.$ We can also assume that $y\neq0$ otherwise
we have the claim $x=y=0.$ From the first equation we obtain%
\[
a_{2i}=a_{2i-1}=0,\quad i=2,\cdots,l+1
\]
So only $a_{2l+3}$ might be non-zero. However since $y\neq0$ we deduce,
appealing to the second equation, that $a_{2l+3}=0$ and thus $a=0$ which is
again impossible.\smallskip

\emph{Case 4 :} $xz\neq0.$ From the first equation we deduce that%
\[
a_{2i}=0,\quad i=2,\cdots,l+1
\]
Inserting this in the second equation we get%
\[
a\wedge\left(  y\,\beta+z\,\gamma\right)  =y\,a_{2l+3}\widehat{e^{3}}%
+z\sum_{i=2}^{l+1}\left(  a_{2i-1}\widehat{e^{2i}}\right)  +y\,a_{3}%
\widehat{e^{2l+3}}.
\]
Since $z\neq0,$ we infer that%
\[
a_{2i-1}=0,\quad i=2,\cdots,l+1.
\]
So only $a_{2l+3}$ might be non-zero. However returning to the first equation
we have%
\[
x\,a_{2l+3}=0.
\]
But since $x\neq0,$ we deduce that $a_{2l+3}=0$ and thus $a=0$ which is again
impossible. This settles the case $k$ even. The odd case is handled in a very
similar manner and we leave out the details\smallskip
\end{proof}

We may now conclude with the proof of Theorem \ref{Thm comme Sverak}, which
is, once the above lemma is established, almost identical to the proof of
Sverak.\smallskip

\begin{proof}
It is enough to prove the theorem for $n=k+3.$ \smallskip

\emph{Step 1.} We start with some notations. Let $L$ be as in Lemma
\ref{Lemme algebrique Sverak}. An element $\xi$ of $L$ is, when convenient,
denoted by $\xi=\left(  x,y,z\right)  \in L.$ Recall that if $\xi=\left(
x,y,z\right)  \in L$ is $1$-divisible, meaning that $\xi=b\wedge a$ for a
certain $a\in\Lambda^{1}$ and $b\in\Lambda^{k-1},$ then necessarily%
\[
xy=xz=yz=0.
\]
We next let $P:\Lambda^{k}\left(  \mathbb{R}^{k+3}\right)  \rightarrow L$ be
the projection map; in particular $P\left(  \xi\right)  =\xi$ if $\xi\in
L.\smallskip$

\emph{Step 2.} Let $g:L\subset\Lambda^{k}\left(  \mathbb{R}^{k+3}\right)
\rightarrow\mathbb{R}$ be defined by%
\[
g\left(  \xi\right)  =-x\,y\,z.
\]
Observe that, $g$ is ext. one affine when restricted to $L.$ Indeed, if
$\xi=\left(  x,y,z\right)  \in L$ and $\eta=\left(  a,b,c\right)  \in L$ is
$1$-divisible (which implies that $ab=ac=bc=0$), then%
$$
g\left(  \xi+t\eta\right)=-\left(  x+ta\right)  \left(  y+tb\right)
\left(  z+tc\right)=-x\,y\,z-t\left[  x\,y\,c+x\,z\,b+y\,z\,a\right]  .
$$
We therefore have that, for every $\xi,\eta\in L$ with $\eta$ being $1$-divisible,%
\[
L_{g}\left(  \xi,\eta\right)  =\left.  \frac{d^{2}}{dt^{2}}g\left(  \xi
+t\eta\right)  \right\vert _{t=0}=0.
\]

\emph{Step 3.} Let $\omega\in C_{per}^{\infty}\left(  \left(  0,2\pi\right)
^{k+3};\Lambda^{k-1}\right)  $ be defined by%
\[
\omega=\left(  \sin x_{1}\right)  \alpha+\left(  \sin x_{2}\right)
\beta+\left(  \sin\left(  x_{1}+x_{2}\right)  \right)  \gamma,
\]
so that
\[
d\omega=\left(  \cos x_{1}\right)  dx^{1}\wedge\alpha+\left(  \cos
x_{2}\right)  dx^{2}\wedge\beta+\left(  \cos\left(  x_{1}+x_{2}\right)
\right)  \left(  dx^{1}+dx^{2}\right)  \wedge\gamma,
\]
and hence $d\omega\in L.$ Note that%
\[
\int_{0}^{2\pi}\int_{0}^{2\pi}g\left(  d\omega\right)  dx_{1}\,dx_{2}%
=-\int_{0}^{2\pi}\int_{0}^{2\pi}\left(  \cos x_{1}\right)  ^{2}\left(  \cos
x_{2}\right)  ^{2}dx_{1}\,dx_{2}<0.
\]

\emph{Step 4.} Assume, cf. Step 5, that we have shown that for every
$\epsilon>0$, we can find $\gamma=\gamma\left(  \epsilon\right)  >0$ such that%
\[
f_{\epsilon}\left(  \xi\right)  =g\left(  P\left(  \xi\right)  \right)
+\epsilon\left\vert \xi\right\vert ^{2}+\epsilon\left\vert \xi\right\vert
^{4}+\gamma\left\vert \xi-P\left(  \xi\right)  \right\vert ^{2}
\]
is ext. one convex. Then noting that%
\[
f_{\epsilon}\left(  d\omega\right)  =g\left(  d\omega\right)  +\epsilon
\left\vert d\omega\right\vert ^{2}+\epsilon\left\vert d\omega\right\vert ^{4},
\]
we deduce from Step 3 that, for $\epsilon>0$ small enough%
\[
\int_{\left(  0,2\pi\right)  ^{k+3}}f_{\epsilon}\left(  d\omega\right)  dx<0.
\]
This shows that $f_{\epsilon}$ is not ext. quasiconvex. The proposition is
therefore proved.$\smallskip$

\emph{Step 5.} It remains to prove that for every $\epsilon>0$ we can find
$\gamma=\gamma\left(  \epsilon\right)  >0$ such that%
\[
f_{\epsilon}\left(  \xi\right)  =g\left(  P\left(  \xi\right)  \right)
+\epsilon\left\vert \xi\right\vert ^{2}+\epsilon\left\vert \xi\right\vert
^{4}+\gamma\left\vert \xi-P\left(  \xi\right)  \right\vert ^{2}%
\]
is ext. one convex. This is equivalent to showing that, for every $\xi,\eta
\in\Lambda^{k}$ with $\eta$ being $1$-divisible,%
\begin{align*}
L_{f_\epsilon}\left(  \xi,\eta\right)   &  =\left.  \frac{d^{2}}{dt^{2}}f_\epsilon\left(
\xi+t\eta\right)  \right\vert _{t=0}\smallskip\\
&  =L_{g}\left(  P\left(  \xi\right)  ,P\left(  \eta\right)  \right)
+2\epsilon\left\vert \eta\right\vert ^{2}+4\epsilon\left\vert \xi\right\vert
^{2}\left\vert \eta\right\vert ^{2}+8\epsilon\left(  \left\langle \xi
;\eta\right\rangle \right)  ^{2}+2\gamma\left\vert \eta-P\left(  \eta\right)
\right\vert ^{2}\smallskip\\
&  \geq0.
\end{align*}
The proof follows in the standard way, see \cite{Dacorogna 2007} and \cite{Sil} for more detail.
\end{proof}

\subsection{Some further examples}

We here give another counterexample for $k=2.$

\begin{proposition}
\label{Prop k=2 quasic n'implique pas polyc}Let $n\geq4.$ Then there exists an
ext. quasiconvex function $f:\Lambda^{2}\left(  \mathbb{R}^{n}\right)\rightarrow\mathbb{R}$
which is not ext. polyconvex.
\end{proposition}

\begin{remark}
This example is mostly interesting when $n=4$ or $5.$ Since when $n\geq6,$ we
already have such a counterexample (cf. Theorem
\ref{Thm contre exemple quadratique k=2}).
\end{remark}

\begin{proof}
As in previous theorems, it is easy to see that it is enough to establish the
theorem for $n=4.$ Let $1<p<2,$ $\alpha=e^{1}\wedge e^{2}+e^{3}\wedge e^{4}$
and $g:\Lambda^{2}\left(  \mathbb{R}^{4}\right)  \rightarrow\mathbb{R}$ be
given by%
\[
g\left(  \xi\right)  =\left(  \left\vert \xi\right\vert ^{2}-2\left\vert
\left\langle \alpha;\xi\right\rangle \right\vert +\left\vert \alpha\right\vert
^{2}\right)  ^{p/2}=\min\left\{  \left\vert \xi-\alpha\right\vert
^{p},\left\vert \xi+\alpha\right\vert ^{p}\right\}  .
\]
The claim is that $f=Q_{ext}g$ has all the desired properties (the proof is
inspired by the one of Sverak \cite{Sverak 1991}, see also Theorem 5.54 in
\cite{Dacorogna 2007}). Indeed $f$ is by construction ext. quasiconvex and if
we can show (cf. Step 2) that $f$ is not convex (note that $f$ is subquadratic and using Proposition \ref{Proposition equiv polyconvexite}, any subquadratic ext. polyconvex function is convex) we will have established the proposition.\smallskip

\emph{Step 1.} First observe that a direct computation gives%
\[
\left\vert \xi\right\vert ^{2}-2\left\vert \left\langle \alpha;\xi
\right\rangle \right\vert +\left\vert \alpha\right\vert ^{2}=\min\left\{
\left\vert \xi-\alpha\right\vert ^{2},\left\vert \xi+\alpha\right\vert
^{2}\right\}  \geq\frac{1}{2}\left[  \left\vert \xi\right\vert ^{2}-\frac
{1}{2}\left\langle \alpha\wedge\alpha;\xi\wedge\xi\right\rangle \right]
\geq0.
\]
We therefore get that there exists a constant $c_{1}>0$ such that%
\begin{align*}
g\left(  \xi\right)   &  \geq\left(\frac{1}{2}\right)^{\frac{p}{2}}\left[  \left\vert \xi\right\vert
^{2}-\frac{1}{2}\left\langle \alpha\wedge\alpha;\xi\wedge\xi\right\rangle
\right]  ^{p/2}\smallskip\\
&  \geq c_{1}\left[  \left\vert \xi_{12}-\xi_{34}\right\vert ^{p}+\left\vert
\xi_{13}+\xi_{24}\right\vert ^{p}+\left\vert \xi_{14}-\xi_{23}\right\vert
^{p}\right]  .
\end{align*}
Call $h$ the right hand side, namely%
\[
h\left(  \xi\right)  =c_{1}\left[  \left\vert \xi_{12}-\xi_{34}\right\vert
^{p}+\left\vert \xi_{13}+\xi_{24}\right\vert ^{p}+\left\vert \xi_{14}-\xi
_{23}\right\vert ^{p}\right]  .
\]

\emph{Step 2.} To prove that $f$ is not convex, we proceed by contradiction. Let us suppose, on the contrary, that $f$ is convex. This implies that $f(0)=0$, because
\[
0\leq f\left(  0\right)  =f\left(  \frac{1}{2}\alpha+\frac{1}{2}\left(
-\alpha\right)  \right)  \leq\frac{1}{2}f\left(  \alpha\right)  +\frac{1}%
{2}f\left(  -\alpha\right)  =0.
\]
Use Remark \ref{21.6.2013.3} to find a
sequence of $\omega_{s}\in W_{\delta,T}^{1,\infty}\left(  \Omega;\Lambda
^{1}\right)  $ (we can choose an $\Omega$ with smooth boundary and by density
we can also assume that $\omega_{s}\in C_{\delta,T}^{\infty}\left(
\overline{\Omega};\Lambda^{1}\right)  $) such that%
\[
0\leq\frac{1}{\operatorname*{meas}\Omega}\int_{\Omega}g\left(  d\omega
_{s}\right)  \leq Q_{ext}g\left(  0\right)  +\frac{1}{s}=f\left(  0\right)
+\frac{1}{s}=\frac{1}{s},
\]
which implies that
\begin{equation}\label{21.6.2013.2}
\lim_{s\rightarrow\infty}\frac{1}{\operatorname*{meas}\Omega}\int_{\Omega}g\left(  d\omega
_{s}\right)=0.
\end{equation}
On the other hand, from Step 1, we deduce that%
\[
0\leq\int_{\Omega}h\left(  d\omega_{s}\right)  \leq\frac{\operatorname*{meas}%
\Omega}{s}\rightarrow0.
\]
We now invoke Step 3 to find a constant $c_{2}>0$ such that%
\[
c_{2}\left\Vert \nabla\omega_{s}\right\Vert _{L^{p}}^{p}\leq\int_{\Omega
}h\left(  d\omega_{s}\right)  .
\]
Thus $\left\Vert d\omega_{s}\right\Vert _{L^{p}}\rightarrow0$ and hence, up to
the extraction of a subsequence,%
\[
\frac{1}{\operatorname*{meas}\Omega}\int_{\Omega}g\left(  d\omega_{s}\right)
\rightarrow g\left(  0\right)  =\left\vert \alpha\right\vert ^{p}\neq0,
\]
which contradicts Equation (\ref{21.6.2013.2}). Therefore, $f$ is not convex.\smallskip

\emph{Step 3.} It remains to prove that there exists a constant $\lambda>0$
such that%
\[
\lambda\left\Vert \nabla\omega\right\Vert _{L^{p}}^{p}\leq\int_{\Omega
}h\left(  d\omega\right)  =\left\Vert \left[  h\left(  d\omega\right)
\right]  ^{1/p}\right\Vert^p _{L^{p}}\,,\quad\text{for every }\omega\in
C_{\delta,T}^{\infty}\left(  \overline{\Omega};\Lambda^{1}\right)  .
\]
To establish the estimate we proceed as follows. Let $\omega\in C_{\delta
,T}^{\infty}\left(  \overline{\Omega};\Lambda^{1}\right)  ,$ $\alpha
,\beta,\gamma\in C^{\infty}\left(  \overline{\Omega}\right)  $ be such that%
\[%
\begin{array}
[c]{ccccc}%
\alpha & = & \left(  d\omega\right)  _{12}-\left(  d\omega\right)  _{34} & = &
-\omega_{x_{2}}^{1}+\omega_{x_{1}}^{2}+\omega_{x_{4}}^{3}-\omega_{x_{3}}%
^{4},\smallskip\\
\beta & = & \left(  d\omega\right)  _{13}+\left(  d\omega\right)  _{24} & = &
-\omega_{x_{3}}^{1}+\omega_{x_{1}}^{3}-\omega_{x_{4}}^{2}+\omega_{x_{2}}%
^{4},\smallskip\\
\gamma & = & \left(  d\omega\right)  _{14}-\left(  d\omega\right)  _{23} & = &
-\omega_{x_{4}}^{1}+\omega_{x_{1}}^{4}+\omega_{x_{3}}^{2}-\omega_{x_{2}}%
^{3},\smallskip\\
0 & = & \delta\omega & = & \omega_{x_{1}}^{1}+\omega_{x_{2}}^{2}+\omega
_{x_{3}}^{3}+\omega_{x_{4}}^{4}\,.
\end{array}
\]
Note that%
\[
h\left(  d\omega\right)  =c_{1}\left[  \left\vert \alpha\right\vert
^{p}+\left\vert \beta\right\vert ^{p}+\left\vert \gamma\right\vert
^{p}\right]  .
\]
Differentiating appropriately the four equations we find
\[
\left\{
\begin{array}{rlrl}
\Delta\omega^{1}  &  =-\alpha_{x_{2}}-\beta_{x_{3}}-\gamma_{x_{4}},&\Delta\omega^{2} &=\alpha_{x_{1}}-\beta_{x_{4}}+\gamma_{x_{3}},\\
\Delta\omega^{3}  &  =\alpha_{x_{4}}+\beta_{x_{1}}-\gamma_{x_{2}},&\Delta\omega^{4}  &  =-\alpha_{x_{3}}+\beta_{x_{2}}+\gamma_{x_{1}}.
\end{array}
\right.
\]
Letting%
\begin{align*}
\psi:=&-\left(  \alpha_{x_{2}}+\beta_{x_{3}}+\gamma_{x_{4}}\right)dx^1+\left(\alpha_{x_{1}}-\beta_{x_{4}}+\gamma_{x_{3}}\right)dx^2+\left(\alpha_{x_{4}}+\beta_{x_{1}}-\gamma_{x_{2}}
\right)dx^3\\&+\left(-\alpha_{x_{3}}+\beta_{x_{2}}+\gamma_{x_{1}}\right)dx^4,
\end{align*}
we get%
\[
\left\{
\begin{array}
[c]{cl}%
\Delta\omega=\psi & \text{in }\Omega\smallskip\\
\nu\wedge\delta\omega=0,\,\nu\wedge\omega=0 & \text{on }\partial\Omega.\smallskip\\
\end{array}
\right.
\]
Using classical elliptic regularity theory, see, for example, Theorem 6.3.7 of \cite{Morrey 1966}, we deduce that
\[
\left\Vert \omega\right\Vert _{W^{1,p}}\leq\lambda_{2}\left\Vert
\psi\right\Vert _{W^{-1,p}}.
\]
In other words,%
\[
\left\Vert \nabla\omega\right\Vert _{L^{p}}\leq\lambda_{2}\left\Vert
\psi\right\Vert _{W^{-1,p}}\leq\lambda_{3}\left\Vert \left(  \alpha
,\beta,\gamma\right)  \right\Vert _{L^{p}}\leq\lambda_{4}\left\Vert \left[
h\left(  d\omega\right)  \right]  ^{1/p}\right\Vert _{L^{p}}\,.
\]
This is exactly what had to be proved.\smallskip
\end{proof}

\section{Application to a minimization problem}

\begin{theorem}
\label{Thm existence de minima}Let $1\leq k\leq n,$ $p>1,$ $\Omega
\subset\mathbb{R}^{n}$ be a bounded smooth open set, $\omega_{0}\in
W^{1,p}\left(  \Omega;\Lambda^{k-1}\right)  $ and $f:\Lambda^{k}\left(
\mathbb{R}^{n}\right)  \rightarrow\mathbb{R}$ be ext. quasiconvex verifying,
\[
c_{1}\left(  \left\vert \xi\right\vert ^{p}-1\right)  \leq f\left(
\xi\right)  \leq c_{2}\left(  \left\vert \xi\right\vert ^{p}+1\right),\text{ for every }\xi\in\Lambda^{k},
\]
for some $c_{1}\,,c_{2}>0.$ Let%
\[
\mathcal{(P}_{0}\mathcal{)\qquad}\inf\left\{  \int_{\Omega}f\left(
d\omega\right)  :\omega\in\omega_{0}+W_{0}^{1,p}\left(  \Omega;\Lambda
^{k-1}\right)  \right\}  =m.
\]
Then the problem $\mathcal{(\mathcal{P}}_{0}\mathcal{)}$ has a minimizer.
\end{theorem}

\begin{remark}
\textbf{(i)} If
\[
\mathcal{(P}_{\delta,T}\mathcal{)\qquad}\inf\left\{  \int_{\Omega}f\left(
d\omega\right)  :\omega\in\omega_{0}+W_{\delta,T}^{1,p}\left(  \Omega
;\Lambda^{k-1}\right)  \right\}  =m_{\delta,T},
\]
where $\omega_{0}+W_{\delta,T}^{1,p}\left(  \Omega;\Lambda
^{k-1}\right)  $ stands for the set of all $\omega\in W^{1,p}\left(
\Omega;\Lambda^{k-1}\right)  $ such that%
\[
\delta\omega=0\text{ in }\Omega\quad\text{and}\quad\nu\wedge\omega=\nu
\wedge\omega_{0}\text{ on }\partial\Omega,
\]
the proof of the theorem will show that $\mathcal{(P}_{\delta,T}\mathcal{)}$
also has a minimizer and that $m_{\delta,T}=m.$\smallskip

\textbf{(ii)} When the function $f$ is not ext. quasiconvex, in general the
problem will not have a solution. However, in many cases it does have one, but
the argument is of a different nature and uses results on differential
inclusions, see \cite{BBDM}, \cite{Band-Dac-Kneuss 2013} and
\cite{Dacorogna-Fonseca}.
\end{remark}

\begin{proof}
\emph{Step 1.} Using a variant of the classical result, see \cite{Band-Sil} and \cite{Sil}, we note that if
\[
\alpha_{s}\rightharpoonup\alpha\quad\text{in }W^{1,p}\left(  \Omega
;\Lambda^{k-1}\right),
\]
then%
\[
\underset{s\rightarrow\infty}{\lim\inf}\int_{\Omega}f\left(  d\alpha
_{s}\right)  \geq\int_{\Omega}f\left(  d\alpha\right)  .
\]

\emph{Step 2.} Let $\omega_{s}$ be a minimizing sequence of $\mathcal{(P}%
_{0}\mathcal{)},$ i.e.%
\[
\int_{\Omega}f\left(  d\omega_{s}\right)  \rightarrow m.
\]
In view of the coercivity condition, we find that there exists a constant
$c_{3}>0$ such that%
\[
\left\Vert d\omega_{s}\right\Vert _{L^{p}}\leq c_{3}\,.
\]

(i) According to Theorem 7.2 in \cite{Csato-Dac-Kneuss 2012} (when $p\geq 2$) and \cite{Sil} (when $p>1$), we can find
$\alpha_{s}\in\omega_{0}+W_{\delta,T}^{1,p}\left(  \Omega;\Lambda
^{k-1}\right)  $ such that%
\[
\left\{
\begin{array}
[c]{cl}%
d\alpha_{s}=d\omega_{s} & \text{in }\Omega\smallskip\\
\delta\alpha_{s}=0 & \text{in }\Omega\smallskip\\
\nu\wedge\alpha_{s}=\nu\wedge\omega_{s}=\nu\wedge\omega_{0} & \text{on
}\partial\Omega
\end{array}
\right.
\]
and there exist constants $c_{4}\,,c_{5}>0$ such that%
\[
\left\Vert \alpha_{s}\right\Vert _{W^{1,p}}\leq c_{4}\left[  \left\Vert
d\omega_{s}\right\Vert _{L^{p}}+\left\Vert \omega_{0}\right\Vert _{W^{1,p}%
}\right]  \leq c_{5}\,.
\]

(ii) Therefore, up to the extraction of a subsequence that we do not relabel,
there exists $\alpha\in\omega_{0}+W_{\delta,T}^{1,p}\left(  \Omega
;\Lambda^{k-1}\right)  $ such that
\[
\alpha_{s}\rightharpoonup\alpha,\quad\text{in }W^{1,p}\left(  \Omega
;\Lambda^{k-1}\right)  .
\]

(iii) We then use Theorem 8.16 in \cite{Csato-Dac-Kneuss 2012} (when $p\geq 2$) and \cite{Sil} (when $p>1$), to find
$\omega\in \omega_0+W_{0}^{1,p}\left(  \Omega;\Lambda^{k-1}\right)  $ such that%
\[
\left\{
\begin{array}
[c]{cl}%
d\omega=d\alpha & \text{in }\Omega\smallskip\\
\omega=\omega_{0} & \text{on }\partial\Omega.
\end{array}
\right.
\]

\emph{Step 3.} We combine the two steps to get%
\[
m=\underset{s\rightarrow\infty}{\lim\inf}\int_{\Omega}f\left(  d\omega
_{s}\right)  =\underset{s\rightarrow\infty}{\lim\inf}\int_{\Omega}f\left(
d\alpha_{s}\right)  \geq\int_{\Omega}f\left(  d\alpha\right)  =\int_{\Omega
}f\left(  d\omega\right)  \geq m.
\]
This concludes the proof of the theorem.
\end{proof}

\section{Notations}

We gather here the notations which we will use throughout this article. For
more details on exterior algebra and differential forms, see
\cite{Csato-Dac-Kneuss 2012} and for the notions of convexity used in the
calculus of variations, see \cite{Dacorogna 2007}.

\begin{enumerate}
\item Let $k,n$ be two integers.

\begin{itemize}
\item We write $\Lambda^{k}\left(  \mathbb{R}^{n}\right)  $ (or simply
$\Lambda^{k}$) to denote the vector space of all alternating $k$-linear maps
$f:\underbrace{\mathbb{R}^{n}\times\cdots\times\mathbb{R}^{n}}_{k-\text{times}%
}\rightarrow\mathbb{R}.$ For $k=0,$ we set $\Lambda^{0}\left(  \mathbb{R}%
^{n}\right)  =\mathbb{R}.$ Note that $\Lambda^{k}\left(  \mathbb{R}%
^{n}\right)  =\{0\}$ for $k>n$ and, for $k\leq n,$ $\operatorname{dim}\left(
\Lambda^{k}\left(  \mathbb{R}^{n}\right)  \right)  ={\binom{{n}}{{k}}}.$

\item $\wedge,$ $\lrcorner\,,$ $\left\langle ;\right\rangle $ and $\ast$ denote the exterior product, the interior product, the
scalar product and the Hodge star operator respectively.

\item If $\left\{  e^{1},\cdots,e^{n}\right\}  $ is a basis of $\mathbb{R}%
^{n},$ then, identifying $\Lambda^{1}$ with $\mathbb{R}^{n},$%
\[
\left\{  e^{i_{1}}\wedge\cdots\wedge e^{i_{k}}:1\leq i_{1}<\cdots<i_{k}\leq
n\right\}
\]
is a basis of $\Lambda^{k}.$ An element $\xi\in\Lambda^{k}\left(
\mathbb{R}^{n}\right)  $ will therefore be written as%
\[
\xi=\sum_{1\leq i_{1}<\cdots<i_{k}\leq n}\xi_{i_{1}i_{2}\cdots i_{k}%
}\,e^{i_{1}}\wedge\cdots\wedge e^{i_{k}}=\sum_{I\in\mathcal{T}_{k}^{n}}\xi
_{I}\,e^{I},
\]
where%
\[
\mathcal{T}_{k}^{n}=\left\{  I=\left(  i_{1}\,,\cdots,i_{k}\right)
\in\mathbb{N}^{k}:1\leq i_{1}<\cdots<i_{k}\leq n\right\}  .
\]
\item We write%
\[
e^{i_{1}}\wedge\cdots\wedge\widehat{e^{i_{s}}}\wedge\cdots\wedge e^{i_{k}%
}=e^{i_{1}}\wedge\cdots\wedge e^{i_{s-1}}\wedge e^{i_{s+1}}\wedge\cdots\wedge
e^{i_{k}}.
\]

\end{itemize}

\item Let $\Omega\subset\mathbb{R}^{n}$ be a bounded open set.

\begin{itemize}
\item The spaces $C^{1}\left(  \Omega;\Lambda^{k}\right)  ,$ $W^{1,p}\left(
\Omega;\Lambda^{k}\right)  $ and $W_{0}^{1,p}\left(  \Omega;\Lambda
^{k}\right)  ,$ $1\leq p\leq\infty$ are defined in the usual way.

\item For $\omega\in W^{1,p}\left(  \Omega;\Lambda^{k}\right)  ,$ the exterior
derivative $d\omega$ belongs to $L^{p}(\Omega;\Lambda^{k+1})$ and is defined
by
\[
(d\omega)_{i_{1}\cdots i_{k+1}}=\sum_{j=1}^{k+1}\left(  -1\right)  ^{j+1}%
\frac{\partial\omega_{i_{1}\cdots i_{j-1}i_{j+1}\cdots i_{k+1}}}{\partial
x_{i_{j}}}\,,
\]
for $1\leq i_{1}<\cdots<i_{k+1}\leq n.$ If $k=0,$ then $d\omega\simeq
\operatorname{grad}\omega.$ If $k=1,$ for $1\leq i<j\leq n,$
\[
(d\omega)_{ij}=\frac{\partial\omega_{j}}{\partial x_{i}}-\frac{\partial
\omega_{i}}{\partial x_{j}}\,,
\]
i.e. $d\omega\simeq\operatorname{curl}\omega.$

\item The interior derivative (or codifferential) of $\omega\in W^{1,p}\left(
\Omega;\Lambda^{k}\right)  ,$ denoted $\delta\omega,$ belongs to $L^{p}%
(\Omega;\Lambda^{k-1})$ and is defined as%
\[
\delta\omega=\left(  -1\right)  ^{n\left(  k-1\right)  }\ast\left(  d\left(
\ast\omega\right)  \right)  .
\]

\end{itemize}
\end{enumerate}

\noindent\textit{Acknowledgments.}
Part of this work was completed during visits of S. Bandyopadhyay to EPFL, whose hospitality and support is gratefully acknowledged. The research of S. Bandyopadhyay was partially supported by a SERB research project titled ``Pullback Equation for Differential Forms".

\end{document}